\RequirePackage{fix-cm}
\documentclass[smallextended]{svjour3}       
\smartqed  

\usepackage{graphicx}      

\usepackage{lmodern,enumerate,rotating,soul}
\usepackage{amsmath,amsfonts,amssymb,subfigure,mathtools,blkarray}
\usepackage{tikz-cd}
\usetikzlibrary{cd}
\usetikzlibrary{shapes,shapes.geometric,arrows,fit,calc,positioning,automata}

\DeclareMathOperator{\Acc}{Acc}
\DeclareMathOperator{\CCa}{CC_A}

\DeclareMathOperator{\CCadia}{CC_A^{\diamond}}

\allowdisplaybreaks

\newcommand{\Z}{\mathbb{Z}}

\newcommand{\N}{\mathbb{N}}

\newcommand{\LM}{\mathcal{L}}

\newcommand{\A}{\mathcal{A}}
\newcommand{\dt}{\delta}

\newcommand{\ep}{\epsilon}

\newcommand{\Scal}{\mathcal{S}}

\newcommand{\Sig}{\Sigma}
\newcommand{\s}{\sigma}

\newcommand{\Mt}{\mathcal{M}}

\newcommand{\nsf}{\mathsf{n}}
\newcommand{\Osf}{\mathsf{O}}
\newcommand{\PSPACE}{\mathsf{PSPACE}}
\newcommand{\coPSPACE}{\mathsf{coPSPACE}}
\newcommand{\NPSPACE}{\mathsf{NPSPACE}}

\newcommand{\PTIME}{\mathsf{PTIME}}
\newcommand{\EXPSPACE}{\mathsf{EXPSPACE}}
\newcommand{\NL}{\mathsf{NL}}

\newcommand{\NP}{\mathsf{NP}}
\newcommand{\coNP}{\mathsf{coNP}}
\newcommand{\SAT}{\mathsf{SAT}}
\newcommand{\PATH}{\mathsf{PATH}}

\newtheorem{assumption}{Assumption}

\newtheorem{remark}{Remark}

\newtheorem{example}{Example}
\usepackage{times,amsmath,amssymb,graphicx,tikz,algorithm,algorithmic,url,hhline,booktabs}
\usetikzlibrary{automata,arrows,positioning}
\usetikzlibrary{petri}

\makeatletter
\newcommand{\subalign}[1]{%
  \vcenter{%
    \Let@ \restore@math@cr \default@tag
    \baselineskip\fontdimen10 \scriptfont\tw@
    \advance\baselineskip\fontdimen12 \scriptfont\tw@
    \lineskip\thr@@\fontdimen8 \scriptfont\thr@@
    \lineskiplimit\lineskip
    \ialign{\hfil$\m@th\scriptstyle##$&$\m@th\scriptstyle{}##$\crcr
      #1\crcr
    }%
  }
}
\makeatother

\usepackage[colorlinks]{hyperref}

\definecolor{green}{rgb}{0.1,0.7,0.1}

\tikzset{
node distance=3cm, 
every state/.style={thick, fill=gray!10}, 
initial text=$ $, 
}

\begin{document}

\tikzset{elliptic state/.style={draw,ellipse}}

\title{A unified method to decentralized state inference and fault diagnosis/prediction of discrete-event systems
}


\author{Kuize Zhang 
}



\institute{K. Zhang \at
			  School of Electrical Engineering and Computer Science, KTH Royal Institute of Technology, 10044 Stockholm, Sweden\\
              \email{kuzhan@kth.se}           
}
%

\maketitle

\begin{abstract}
	The state inference problem and fault diagnosis/prediction problem are fundamental topics
	in many areas. In this paper, we consider discrete-event systems (DESs) modeled by
	finite-state automata (FSAs).
	There exist plenty of results on decentralized
	versions of the latter problem but there is almost no result for a decentralized version of
	the former problem. In this paper, we propose a decentralized version of strong detectability
	called {\it co-detectability} which implies that 
	once a system satisfies this property, for each generated infinite-length event sequence,
	at least one local observer can determine the current and subsequent states
	after a common observation time delay.
	We prove that the problem of verifying co-detectability of FSAs is $\coNP$-hard.
	Moreover, we use a unified {\it concurrent-composition} method to give $\PSPACE$ verification
	algorithms for co-detectability, co-diagnosability, and co-predictability of FSAs,
	without any assumption or modifying the FSAs under consideration, where 
	co-diagnosability is firstly studied by [Debouk \& Lafortune \& Teneketzis 2000],
	while co-predictability is firstly studied by [Kumar \& Takai 2010].
	By our proposed unified method,
	one can see that in order to verify co-detectability, more
	technical difficulties will be met compared to verifying the other two properties, because
	in co-detectability, generated outputs are counted, but in the latter two properties, 
	only occurrences of events are counted. For example, when one output was generated,
	any number of unobservable events could have occurred. The $\PSPACE$-hardness of verifying 
	co-diagnosability is already known in the literature. In this paper, we prove the
	$\PSPACE$-hardness of verifying co-predictability.
	
\end{abstract}

\section{Background}

	\subsection{State inference}

The state inference problem of dynamical systems has been a central problem in 
computer science \cite{Moore1956,Model-BasedTesting2005}.
This problem has been a central problem also in control theory,
arranging from linear systems \cite{Kalman1963MathDescriptionofLDS,Wonham1985LinearMultiControl}, 
to nonlinear systems \cite{Sontag1979,Conte2007AlgebraicMethodsNonlinearControlSystems,Isidori1999NonConSys}, 
to switched systems \cite{Tanwani2013ObsSwiLinSys}, and also to networked systems
\cite{Kibangou2016ObserNetSys,Angulo2019StrucObserNonNetSys}.
In two seminal papers \cite{Moore1956,Kalman1963MathDescriptionofLDS},
a property (with its variants)
of
whether one can use an input sequence and the corresponding output sequence to determine
the initial state is investigated. In the former, the property is called {\it Gedanken-experiment}
and Moore machines (deterministic finite-state machines, not necessarily linear) are considered;
while in the latter, it is called {\it observability} and linear differential equations are 
considered. Theoretically, such a property is of its intrinsic interest. When internal states
are only partially observed, it is interesting to 
develop different techniques according to features of different models to infer internal states by using
partial observations.
From a practical point of view, such a property has extensive applications in different areas,
e.g., in traffic networks, it is meaningful to locate a crucial car by using traceable interaction
information with the car when the car itself is not directly traceable;
in genetic regulatory networks, 
it is important to use the states of a subset of directly measurable nodes to estimate or
determine the whole network state because usually not all nodes could be directly measured
\cite{Liu2013ObservabilityComplexSystem}. 
When the initial or past state information
is not crucial but only the current and subsequent state information is needed, the property 
could be reformulated as a weaker notion of {\it detectability} 
\cite{Shu2007Detectability_DES,Fornasini2013ObservabilityReconstructibilityofBCN,Zhang2016WPGRepresentationReconBCN,Sandberg2005HomingSynchronizingSequence,Kari2003SynchronizingSequence,Eppstein1990SIAMJCResetSequence},
which means that whether one could determine the current and subsequent states
by using observed information.

{\it Discrete-event systems} (DESs) consists of discrete states and transitions between states caused
by spontaneous occurrences of events \cite{Cassandras2009DESbook,WonhamSupervisoryControl}, where 
states and events are partially observed. DESs could also be regarded as a suitable model for the
cyber-layer of {\it cyber-physical systems} (CPSs) that are ubiquitous in control engineering,
computer technology, communication engineering, etc. Usually, a CPS consists of a cyber layer and a physical
layer, where the former is a decision process that is usually a discrete system, and the latter usually
comprises several physical processes modeled by differential equations. The two layers are connected by
networks, where the cyber layer should be able to monitor the working status of the physical layer in real time,
and also allocate commands to the physical layer, both through networks. In such a way,
DESs play a central role in governing global behavior of CPSs, and the detectability property
of DESs is of particular importance in governing the global behavior.

\subsection{Fault diagnosis/prediction}

As mentioned before, DESs have two partially observed components, states and events. Hence one may 
be interested in inference problems of either states or events by using observed information.
The inference problem to
the former is formulated as detectability as mentioned before. While the inference of occurrences
of several events could be formulated as {\it diagnosability}
\cite{Lin1994DiagnosabilityDES,Sampath1995DiagnosabilityDES}, where if the property 
holds then once a 
special event (usually regarded to be faulty) occurred, after sufficiently many new events occurred,
one could make sure that a faulty event (although may not be the previous one) had occurred. 
This property means inference of occurrences of faulty events. However, from a dual point of view,
sometimes the occurrence of a faulty event will lead to great economic loss, which also
motivates the study of a dual notion of {\it predictability} \cite{Genc2009PredictabilityDES},
where if this notion holds then 
before a particular faulty event occurs, 
one could make sure that some faulty event will definitely occur. The same as detectability, diagnosability
also has extensive applications, e.g., in railway traffic systems \cite{Boussif2019Diagnosability}.

\subsection{Literature review and an idea of unifying detectability and diagnosability/predictability}

In the literature, these properties are treated by using different methods, and most of
the corresponding 
methods rely on (at least one of) TWO FUNDAMENTAL ASSUMPTIONS that a system is deadlock-free 
(which means that it can
always run), and has no unobservable reachable cycle (which means its running can always be observed).
The two assumptions for FSAs are formulated in Assumption \ref{assum1_Det_PN}.
The notions of {\it strong detectability} and {\it weak detectability} are two fundamental notions of detectability
of DESs. The former implies that there is delay $n$, for each infinite-length output sequence generated by a DES,
each prefix of the output sequence of length greater than $n$ allows reconstructing the current state.
While the latter implies that
there is delay $n$, for some infinite-length output sequence generated by a DES,
each prefix of the output sequence of length greater than $n$ allows doing that.
These two notions are first studied in the seminal paper \cite{Shu2007Detectability_DES}.
Under Assumption \ref{assum1_Det_PN}, a polynomial-time verification algorithm
based on a {\it detector} method for strong detectability
and an exponential-time verification algorithm based on an {\it observer} method for weak detectability
are given in \cite{Shu2011GDetectabilityDES} and \cite{Shu2007Detectability_DES}, respectively.
Later, verifying weak detectability is proved to be PSPACE-hard,
even for deterministic fully-observed FSAs 
\cite{Zhang2017PSPACEHardnessWeakDetectabilityDES,Masopust2018ComplexityDetectabilityDES}.
Initial results on detectability of labeled Petri nets could be found in 
\cite{Zhang2018WODESDetectabilityLPS,Masopust2018DetectabilityPetriNet,Zhang2019DetPNFA},
where in \cite{Zhang2018WODESDetectabilityLPS}, weak detectability is proved to be 
undecidable for labeled Petri nets with inhibitor arcs, and later this undecidable result
is strengthened to hold for labeled Petri nets in \cite{Masopust2018DetectabilityPetriNet};
strong detectability is proved to be decidable in \cite{Masopust2018DetectabilityPetriNet}
for labeled Petri nets under the two fundamental assumptions
for labeled Petri nets,
and later the decidable result is strengthened to hold only based on
the second of the two assumptions in \cite{Zhang2020bookDDS}. Results on essentially different
variants of detectability
notions could be found in \cite{Zhang2019DetPNFA}.
In \cite{Zhang2019KDelayStrDetDES}, a {\it concurrent-composition} method is found to 
verify strong detectability
of FSAs in polynomial time without any assumption, where the 
concurrent-composition structure exactly comes from characterizing {\it negation} of
strong detectability. Note that the terminology ``concurrent composition'' is 
not a new one, but already exists in the literature, representing similar operations to
automata or transition systems compared to those in the current paper.



For diagnosability of FSAs, in the seminal paper \cite{Sampath1995DiagnosabilityDES}, the notion is 
first formulated, and an exponential-time verification algorithm based on a {\it diagnoser} method
is given also under Assumption 
\ref{assum1_Det_PN}. Later on, a {\it twin-plant} method is developed respectively in 
\cite{Jiang2001PolyAlgorithmDiagnosabilityDES,Yoo2002DiagnosabiliyDESPTime} so that 
polynomial-time
verification algorithms are designed, but still under Assumption \ref{assum1_Det_PN}.
After comparing the twin-plant structure with the concurrent-composition structure,
one can see that the former is actually a simplified version of the latter.
In more details, 
in the former, unobservable transitions are not synchronized but only common observable
transitions are synchronized; while in the latter, observable transitions are synchronized as pairs, and
unobservable transitions are also synchronized, so 
more information is contained. Later on, the twin-plant method has been extended so that it
does not rely on Assumption \ref{assum1_Det_PN} any more, e.g., 
in \cite{Moreira2011Codiagnosability}.
On the other hand, there is an alternative verification algorithm for diagnosability given in 
\cite{Viana2019CodiagnosabilityDES}
without any assumption, but the algorithm runs in exponential time, because it is based on the 
observer proposed in \cite{Shu2007Detectability_DES} that is of exponential complexity.
Further related works 
could be found in \cite{Ye2011PhDThesisDiagnosability,Su2004PhDTheseDiagnosisDES,Ye2009PatternDiagnosabilityDistDES,Ye2013JointDiagnosabilityDES,Ibrahim2017DiagnosabilityPlanningConDES,Cassez2008FaultDiagnosisStDyObser},
etc.
The notion of predictability is first proposed in \cite{Genc2009PredictabilityDES}, in which a 
polynomial-time verification algorithm is designed, under \ref{item12_Det_PN}) of Assumption \ref{assum1_Det_PN}.
Differently from detectability, there exist a large number of publications on diagnosability and predictability
with their variants, it is partially because in the seminal paper \cite{Sampath1995DiagnosabilityDES},
it is not defined what diagnosability is for a terminating transition sequence.
We can introduce only a few of them.
We refer the reader to \cite{Hadjicostis2020DESbook} for more related references.
For results on diagnosability of labeled Petri nets, we refer the reader to 
\cite{Cabasino2012DiagnosabilityPetriNet,Berard2018DiagnosabilityPetriNet,Yin2017DiagnosabilityLabeledPetriNets,Haar2012DiagnosabilityTopologyPetriNet}, etc.
In \cite{Cabasino2012DiagnosabilityPetriNet}, a new technique is developed to
verify diagnosability; in \cite{Yin2017DiagnosabilityLabeledPetriNets}, diagnosability is proved
to be decidable with $\EXPSPACE$ lower bound under the first of the two fundamental assumptions;
in \cite{Berard2018DiagnosabilityPetriNet},
a weaker notion of diagnosability
called trace diagnosability is proved to have $\EXPSPACE$ upper bound and lower bound
without any assumption.

Sometimes, limited by the ability of external observers, not all observable events could be observed,
which weakens the possibility of determining states or occurrences of events. In order to deal with such 
a setting, a {\it decentralized version} of the above properties is investigated.
In a decentralized version,
one chooses several observers and put them into different places, and takes into accounts the results returned
by all local observers and makes a final verification whether the system satisfies a property.
In such a way, the original version of a property could be call the {\it centralized version}.
A {\it decentralized version of diagnosability} called {\it co-diagnosability}
of FSAs is firstly studied in
\cite{Debouk2000CodiagnosabilityAutomata} and later revisited in
\cite{Qiu2006DecentralizedFD,Wang2007DecentralizeDiagnosisDES}, etc., by extending
the original twin-plant structure. The results in 
\cite{Qiu2006DecentralizedFD,Wang2007DecentralizeDiagnosisDES} hardly rely on Assumption
\ref{assum1_Det_PN}, while it is clearly shown in \cite{Moreira2011Codiagnosability} that
a simplified version of the extended twin-plant structures used in 
\cite{Qiu2006DecentralizedFD,Wang2007DecentralizeDiagnosisDES}
does not rely on any assumption. Actually, the verification algorithm
shown in \cite{Moreira2011Codiagnosability} runs in exponential time, but not in polynomial
time as is claimed.
The problems of verifying decentralized versions of diagnosability 
for deterministic FSAs are proved to be $\PSPACE$-hard
\cite{Cassez2012ComplexityCodiagnosability}, and a
$\PSPACE$ upper bound is also given based on item \ref{item11_Det_PN}) of Assumption \ref{assum1_Det_PN}
for {\it timed automata} that are substantially more general than FSAs.

The results on decentralized versions of predictability of FSAs
could be found in \cite{Liu2019Predictability,Kumar2010CoPrognosisDES}, where the notion of 
co-predictability studied in \cite{Liu2019Predictability} is exactly the decentralized version of the 
predictability
proposed in \cite{Genc2009PredictabilityDES}, and is equivalent to the notion of uniform bounded
prognosability studied in \cite{Kumar2010CoPrognosisDES}.
The results in 
\cite{Liu2019Predictability} are based on Assumption \ref{assum1_Det_PN};
but in \cite{Kumar2010CoPrognosisDES}, the results work 
after adding at each deadlock state an unobservable self-loop. We will point out that
such a modification does not always preserve (co)-diagnosability or (co)-predictability,
as will shown in Remarks \ref{rem1:codiag} and \ref{rem1:copredic}.
Further results can be found in \cite{Yin2019DecenFaultPrognosis,Ye2013DistributivePredictabilityDES}, etc.
Unlike diagnosability or predictability, there exists almost no result on a decentralized version of 
detectability. 
The only results on decentralized versions of detectability could be found in \cite{Shu2011CoDetDES},
but the notions are not very reasonable,
because they require some local observer to observe all observable events,
so they are actually equivalent to centralized versions of detectability.
In addition, language-based decentralized observability results (called joint observability)
could be found in 
\cite{Tripakis2004DentralizedObservation,Giua2017DecenObservability},
where generally the verification problem is undecidable.
In this paper, we will reformulate a notion of co-detectability that matches a decentralied
setting, and characterize its complexity.

To sum up, there have been plenty of results on decentralized versions of diagnosability
and predictability 
of FSAs, while there has been almost no result on a decentralized version of detectability.
We will formulate a notion of co-detectability,
and extend the concurrent-composition method developed in 
\cite{Zhang2019KDelayStrDetDES,Zhang2019DetPNFA} to
give a method for verifying co-detectability of FSAs, without any assumption or modifying
the FSAs. We will also show that the method also works for co-diagnosability and 
co-predictability after being simplified. Moreover, we will characterize
complexity of the notions of co-detectability, co-diagnosability, and co-predictability.
As potential extensions, these results could be used to study more general distributed 
versions of these notions under weaker assumptions compared to existing results in 
the literature (e.g., in
\cite{Ye2013DistributivePredictabilityDES,Keroglou2018DistributedFaultDiagnosis}).


\subsection{Contributions}

The contributions of the paper are listed as follows (briefly illustrated in Tabs. 
\ref{tab1:ComplexityDecentrFundaPropertyFSA} and \ref{tab2:ComplexityFundaPropertyFSA}
together with related results in the literature).
\begin{enumerate}
	\item We formulate a notion of co-detectability, use the concurrent-composition method to
		give a $\PSPACE$ verification algorithm for the property.
		We also prove that it is $\coNP$-hard to verify the notion
		by reducing the known $\NP$-complete {\it acyclic deterministic finite automata (DFAs)
		intersection problem}
		\cite{Rampersad2010NP-hardAcycDFA} to negation of co-detectability.
	\item For co-diagnosability and co-predictability in the literature, we use the concurrent-composition method to
		give a $\PSPACE$ verification algorithm for both properties. We also prove that it is $\PSPACE$-hard
		to verify co-predictability by reducing the known $\PSPACE$-complete {\it DFAs intersection 
		problem}
		\cite{Kozen1977FiniteAutomatonIntersection} to negation of co-predictability
		(a $\PSPACE$ lower bound for
		co-diagnosability is already given in \cite{Cassez2012ComplexityCodiagnosability}).
\end{enumerate}

\begin{table}
	\centering
	\begin{tabular}{ccc}
		\hline\hline
		co-detectability & co-diagnosability & co-predictability\\
		\hline
		\begin{tabular}{c}
			$\PSPACE$ (Thm. \ref{thm2_codet})\\
			$\coNP$-hard (Thm. \ref{thm3_codet})
		\end{tabular}
		&
		\begin{tabular}{c}
			$\PSPACE$ (Thm. \ref{thm2_codiag})\\
			$\PSPACE$-hard (\cite{Cassez2012ComplexityCodiagnosability})
		\end{tabular}
		&
		\begin{tabular}{c}
			$\PSPACE$ (Thm. \ref{thm2_copred})\\
			$\PSPACE$-hard (Thm. \ref{thm3_copred})
		\end{tabular}\\
		\hline
	\end{tabular}
	\caption{Complexity results for decentralized versions of detectability,
	diagnosability, and predictability of finite-state automata.}
	\label{tab1:ComplexityDecentrFundaPropertyFSA}
\end{table}

\begin{table}
	\centering
	\begin{tabular}{ccc}
		\hline\hline
		strong detectability & diagnosability & predictability\\
		\hline
		\begin{tabular}{c}
			$\NL$ (Thm. \ref{thm4_codet})\\
			$\NL$-hard (\cite{Masopust2018ComplexityDetectabilityDES})
		\end{tabular}
		&
		\begin{tabular}{c}
			$\NL$-complete (\cite{Berard2018DiagnosabilityPetriNet})
		\end{tabular}
		&
		\begin{tabular}{c}
			$\NL$-complete (Thm. \ref{thm4_copred})
		\end{tabular}\\
		\hline
	\end{tabular}
	\caption{Complexity results for strong detectability,
	diagnosability, and predictability of finite-state automata.}
	\label{tab2:ComplexityFundaPropertyFSA}
\end{table}

In the subsequent main results, one will see that in order to characterize co-detectability,
more technical difficulties will be met than to deal with co-diagnosability and co-predictability,
because in co-detectability, generated outputs are counted, but in co-diagnosability and co-predictability,
only generated events are counted. When one output was counted, any number of unobservable events
could have occurred.

The rest of the paper is structured as follows. In Section \ref{sec2:pre}, we introduce 
preliminaries on FSAs and languages, and the notion of concurrent composition of 
FSAs. In Section \ref{sec3:mainresult}, we show main results. Finally in Section \ref{sec4:conc},
we show a short conclusion.

\section{Preliminaries}\label{sec2:pre}

Next we introduce necessary notions that will be used throughout this paper.
Symbols $\N$ and $\Z_{+}$ denote the sets of natural numbers and positive integers, respectively.
For a set $\Sig$, $\Sig^*$ and $\Sig^{\omega}$ are used to denote the sets of finite sequences
(called {\it words}) of elements of $\Sig$ including the empty word $\epsilon$
and infinite sequences (called {\it configurations}) of elements of $\Sig$,
respectively. As usual, we denote $\Sig^{+}=\Sig^*\setminus\{\epsilon\}$.
For a word $s\in \Sig^*$,
$|s|$ stands for its length, and we set $|s'|=+\infty$ for all $s'\in \Sig^{\omega}$.
For $s\in \Sig$ and natural number $k$,
$s^k$ and $s^{\omega}$ denote the $k$-length word and configuration
consisting of copies of $s$'s, respectively.
For a word (configuration) $s\in \Sig^*(\Sig^{\omega})$, a word $s'\in \Sig^*$ is called a {\it prefix} of $s$,
denoted as $s'\sqsubset s$,
if there exists another word (configuration) $s''\in \Sig^*(\Sig^{\omega})$ such that $s=s's''$.
For two natural numbers $i\le j$, $[i,j]$ denotes the set of all integers between $i$ and $j$ including $i$ and $j$;
and for a set $S$, $|S|$ its cardinality and $2^S$ its power set.

\subsection{Finite automata}

A {\it nondeterministic finite automaton} (NFA) is a $5$-tuple 
\begin{equation}\label{NFA}
	\A=(Q,\Sig,\dt,q_0,F),
\end{equation}
where $Q$ is a finite set of {\it states}, $\Sig$ is a finite {\it alphabet},
elements of $\Sig$ are called {\it letters} \cite{Kari2013LectureNoteonAFL} (also called 
{\it events} 
(cf. \cite{Ramadge1987SupervisoryControlofDES,Shu2007Detectability_DES}, etc.)),
$q_0\in Q$ is the {\it initial state}, $F\subset Q$ is the set of {\it accepting states}
(also called {\it final states}
\cite{Kari2013LectureNoteonAFL} or {\it marker states} \cite{Ramadge1987SupervisoryControlofDES}),
$\dt\subset Q\times \Sig\times Q$ is a {\it transition relation}.
A word $\s_1\dots\s_n\in\Sig^*\setminus\{\epsilon\}$ is called {\it accepted by} $\A$
if there exist states $q_1,\dots,q_n\in Q$ such that $q_n\in F$ and $(q_{i},\s_{i+1},q_{i+1})\in\dt$
for all $0\le i\le n-1$. Particularly $\epsilon$ is {\it accepted by} $\A$ if and only if $q_0\in F$.
The set of words accepted by $\A$ is called the {\it language recognized by} $\A$.
Automaton $\A$ is called {\it acyclic} if there is no cycle in $\A$;
called {\it complete} if for each state $q\in Q$ and each letter $\s\in\Sig$, there is a transition
$(q,\s,q')\in\dt$ for some $q'\in Q$;
and called a {\it deterministic finite automaton} (DFA) if for all $q_1,q_2',q_2''\in Q$
and $\s\in\Sig$, $(q_1,\s,q_2')\in\dt$ and $(q_1,\s,q_2'')\in\dt$ imply $q_2'=q_2''$.

\subsection{Finite-state automata}

A DES can be modeled by an FSA, which can be obtained from
an NFA \eqref{NFA} by removing all accepting states,
replacing a unique initial state
by a set $X_0$ of initial states, and adding a labeling function $\ell$.

Formally, an FSA is a sextuple
\begin{equation}\label{DES_FSA}
	\Scal=(X,T,X_0,\dt,\Sig,\ell),
\end{equation}
where $X$ is a finite set of {\it states}, $T$ a finite set of {\it events},
$X_0\subset X$ a set of {\it initial states},
$\dt\subset X\times T\times X$ a {\it transition relation}, $\Sig$ a finite set of {\it outputs (labels)},
and $\ell:T\to\Sig\cup\{\epsilon\}$
a {\it labeling function}, where $\epsilon$ denotes the empty word.
The event set $T$ can been rewritten as disjoint union
of {\it observable} event set $T_{o}$ and {\it unobservable} event set $T_{\ep}$,
where events of $T_o$ are with label in $\Sig$, but events of $T_{\ep}$ are labeled by $\ep$.
When an observable event occurs, its label can be observed, but when an unobservable event occurs, 
nothing can be observed.
For an observable event $t\in T$, we say $t$ {\it can be directly observed} if $\ell(t)$
differs from $\ell(t')$ for any other $t'\in T$. Transition relation $\dt\subset X\times T
\times X$ can be recursively extended to $\dt\subset X\times T^* \times X$ as follows:
(1) for all $x,x'\in X$, $(x,\ep,x')\in\dt$ if and only if $x=x'$; (2) for all $x,x'\in X$,
$s\in T^*$, and $t\in T$, one has $(x,st,x')\in\dt$, also denoted by $x\xrightarrow[]{st}
x'$, called a {\it transition sequence},
if and only if $(x,s,x''),(x'',t,x')\in\dt$ for some $x''\in X$.
Labeling function $\ell:T\to\Sig\cup\{\ep\}$ 
can be recursively extended to $\ell:T^*\cup T^{\omega}\to\Sig^*\cup\Sig^{\omega}$ as
$\ell(t_1t_2\dots)=\ell(t_1)\ell(t_2)\dots$ and $\ell(\ep)=\ep$.
Transitions $x\xrightarrow[]{t}x'$ with $\ell(t)=\ep$ are called 
{\it unobservable transitions}, and other transitions are called
{\it observable transitions}. The event set $T$ can also been rewritten as disjoint union
of {\it controllable} event set $T_{c}$ and {\it uncontrollable} event set $T_{uc}$,
where controllable events are such that one can disable their occurrences, and uncontrollable
events are such that one cannot do that. Analogously, transitions $x\xrightarrow[]{t}x'$
with $t$ being controllable are called {\it controllable}, and other transitions are called
{\it uncontrollable}. For $x\in X$ and $s\in T^+$, $(x,s,x)$ is called a {\it transition 
cycle} if $(x,s,x)\in\dt$. An {\it observable transition cycle} is defined 
by a transition cycle with at least one observable transition. Analogously an {\it unobservable
transition cycle} is defined by a transition cycle with no observable transition.
Automaton $\Scal$ is called {\it deterministic}
if for all $x,x',x''\in X$ and $t\in T$, $(x,t,x'),(x,t,x'')\in\dt$ imply $x'=x''$.

A state $x\in X$ is called {\it deadlock} if 
$(x,t,x')\notin \dt$ for any $t\in T$ and $x'\in X$.
$\Scal$ is called {\it deadlock-free} if it has no deadlock state.
We say {\it a state $x'\in X$ is reachable from a state}
$x\in X$ if there exists $s\in T^+$ such that $x\xrightarrow[]{s}x'$.
We say {\it a subset $X'$ of $X$ is reachable from a state} $x\in X$ if some state of $X'$ is 
reachable from $x$. Similarly {\it a state $x\in X$ is reachable from a subset $X'$ of $X$}
if $x$ is reachable from some state of $X'$.
We call a state $x\in X$
{\it reachable} if either $x\in X_0$ or it is reachable from some initial state.
We denote by $\Acc(\Scal)$ the {\it accessible part} of $\Scal$ that is obtained by 
removing all unreachable states of $\Scal$.

We use $L(\Scal)=\{s\in T^*|(\exists x_0\in X_0)(\exists x\in X)[x_0\xrightarrow[]
{s}x]\}$ to denote the set of finite-length event sequences 
generated by $\Scal$,
we also use $L^{\omega}(\Scal)=\{t_1t_2\dots$$\in T^{\omega}|(\exists x_0\in X_0)
(\exists x_1,x_2,\dots$$\in X)[x_0\xrightarrow[]{t_1}x_1\xrightarrow[]{t_2}\cdots]\}$ to 
denote the set of infinite-length event sequences generated by $\Scal$.
For each $\s\in\Sig^*$, we denote by $\Mt(\Scal,\s)$ the {\it current state estimate},
i.e., set of
states that the system can be in after $\s$ has been observed, i.e., 
$\Mt({\cal S},\s):=
\{x\in X|(\exists x_0\in X_0)(\exists s\in T^*)[
(\ell(s)=\s)\wedge(x_0\xrightarrow[]{s}x)]\}$.
$\LM({\cal S})$ denotes the {\it language generated} by system $\cal S$,
i.e., $\LM({\cal S}):=\{\s\in\Sig^*|\Mt({\cal S},\s)\ne\emptyset\}$.
We use $\LM^{\omega}({\cal S})$ to denote the $\omega$-{\it language generated} by $\cal S$,
i.e., $\LM^{\omega}(\Scal):=\{\s\in\Sigma^{\omega}|(\exists s\in L^{\omega}(\Scal)|
[\ell(s)$$=\s]\}$.


The following two assumptions are commonly used in detectability studies 
(cf. \cite{Shu2007Detectability_DES,Shu2011GDetectabilityDES,Shu2013DelayedDetectabilityDES})
and diagnosability studies \cite{Sampath1995DiagnosabilityDES,Jiang2001PolyAlgorithmDiagnosabilityDES,Yoo2002DiagnosabiliyDESPTime},
but are not needed in the current paper.

\begin{assumption}\label{assum1_Det_PN}
	An FSA $\Scal=(X,T,X_0,\dt,\Sig,\ell)$ satisfies
	\begin{enumerate}[(i)]
		\item\label{item11_Det_PN}
			$\Scal$ is deadlock-free, 
		\item\label{item12_Det_PN} 	
			$\Scal$ is prompt (or divergence-free), 
			i.e., for every reachable state $x\in X$ and every
			nonempty unobservable event sequence $s\in (T_\ep)^+$, $(x,s,x)\notin\dt$.
	\end{enumerate}
\end{assumption}


%

\subsection{Concurrent composition}

The concurrent-composition structure was found in \cite{Zhang2019DetPNFA} when {\it negation}
of stronger versions of detectability was characterized. It provides a polynomial-time verification
method to strong detectability \cite{Zhang2019KDelayStrDetDES,Zhang2020bookDDS} without any assumption,
which strengths 
the detector method for verifying strong detectability proposed in \cite{Shu2011GDetectabilityDES}
under Assumption \ref{assum1_Det_PN}. In this paper, we extend the concurrent composition structure
from two automata to a finite number of automata, and show that it could provide a unified approach
to verifying decentralized versions of strong detectability, diagnosability, and predictability,
without any assumption or changing the automata under consideration.

Given $L\in\Z_{+}$ and a number $L+1$ of FSAs $\Scal_i=(X^i,T^i,X_0^i,\dt^i,\Sig^i,\ell^i)$, $i\in[0,L]$,
we construct the {\it concurrent composition} 
\begin{equation}\label{ConComp_automata}
	\CCa(\Scal_0;\Scal_1,\dots,\Scal_L):=(X',T',X_0',\dt')
\end{equation} {\it of $\Scal_1,\dots,\Scal_L$ with respect to} $\Scal_0$ as follows:
\begin{enumerate}
	\item $X'=X^0\times X^1\times\cdots\times X^L$;
	\item $T'\subset(T^0\cup\{\ep\})\times(T^1\cup\{\ep\})\times\cdots\times(T^L\cup\{\ep\})$,
		where for all $(\breve{t}^0,\breve{t}^1,\dots,\breve{t}^L)\in
		(T^0\cup\{\ep\})\times(T^1\cup\{\ep\})\times\cdots\times(T^L\cup\{\ep\})$,
		$(\breve{t}^0,\breve{t}^1,\dots,\breve{t}^L)\in T'$ if and only if one of the followings
		holds: 
				\begin{enumerate}
			\item\label{item1_codet}
				$\ell^{j}(\breve{t}^0)=\ep$ for all $j\in[1,L]$ such that $\breve{t}^0\in
				T^j\cup\{\ep\}$ (which includes the case $\breve{t}^0\notin\bigcup_{i=1}
				^{L}T^i$): 
				
				$\ell^{k}(\breve{t}^{k})=\ep$ for all $k\in[1,L]$, and
				$\breve{t}^{l}\ne\ep$ for only one $l\in[0,L]$,
			\item\label{item2_codet} 
				$\ell^{j}(\breve{t}^0)\ne\ep$ for some $j\in[1,L]$ such that $\breve{t}^0
				\in T^j$:
				
				for all $l\in[1,L]$, $\breve{t}
				^{l}=\ep$ if $\ell^{l}(\breve{t}^0)=\ep$ or $\breve{t}^0\notin T^l$,
				$\ell^{l}(\breve{t}^{l})=
				\ell^{l}(\breve{t}^0)$ otherwise;
		\end{enumerate} 
	\item $X_0'=X^0_0\times X^1_0\times\cdots\times X^L_0$;
	\item for all $(\breve{x}^0,\breve{x}^1,\dots,\breve{x}^L),(\breve{y}^0,\breve{y}^1,\dots,\breve{y}^L)\in X'$,
		and $(\breve{t}^0,\breve{t}^1,\dots,\breve{t}^L)\in T'$, one has
		$((\breve{x}^0,\breve{x}^1,\dots,\breve{x}^L),(\breve{t}^0,\breve{t}^1,\dots,\breve{t}^L),
		(\breve{y}^0,\breve{y}^1,\dots,\breve{y}^L))\in\dt'$ if and only if
		\begin{itemize}
			\item for all $i\in[0,L]$, $\breve{x}^i=\breve{y}^i$ if $\breve{t}^i=\ep$, 
				$(\breve{x}^i,\breve{t}^i,\breve{y}^i)\in\dt^i$ otherwise.
		\end{itemize}
\end{enumerate}

For all $\breve{t}^0\in T^0$ and $i\in[1,L]$ such that $\breve{t}^0\notin T^i$, 
we denote $\ell^i(\breve{t}^0)=\ep$.
For an event sequence $s'\in (T')^{*}$, we use $s'(i)$ to denote its $i$-th 
component, $i\in[0,L]$. This notation is applied to states of $X'$ and transition sequences 
of \eqref{ConComp_automata} as well. 
Then for all $i\in[1,L]$, one has $\ell^i(s'(0))=\ell^i(s'(i))=:\ell^i(s')$.
Then one observes $\CCa(\Scal_0;\Scal_1,\dots,\Scal_L)$ aggregates
all transition sequences of $\Scal_0,\dots,\Scal_L$ starting from initial states
such that the transition sequences in $\Scal_i$ 
and $\Scal_0$ produce the same label sequence under labeling function $\ell^i$, $i\in[1,L]$.
Hence we call this notion the concurrent composition of $\Scal_1,\dots,\Scal_L$ with respect to 
$\Scal_0$. When $L=1$ and $\Scal_0=\Scal_1$, \eqref{ConComp_automata} reduces to the concurrent
composition used in \cite{Zhang2019DetPNFA,Zhang2020bookDDS}. The above observation could be
formulated as the following proposition, we omit its straightforward proof.

\begin{proposition}\label{prop1_co_DES}
	Consider FSAs $\Scal_i=(X^i,T^i,X_0^i,\dt^i,\Sig^i,\ell^i)$, $i\in[0,L]$,
	and transition sequences $x_0^i\xrightarrow[]{s^i}x^i$ of $\Scal_i$ with $x_0^i$
	being initial and $s^i\in (T^i)^*$. If for all $i\in[1,L]$,
	$\ell^i(s^0)=\ell^{i}(s^i)$, then in concurrent composition $\CCa(\Scal_0;\Scal_1,\dots,\Scal_L)$,
	there is a transition sequence $x_0'\xrightarrow[]{s'}x'$ the $j$-th component of which equals 
	the above $x_0^j\xrightarrow[]{s^j}x^j$ for all $j\in[0,L]$.
\end{proposition}

Concurrent composition $\CCa(\Scal_0;\Scal_1,\dots,\Scal_L)$ has at most
$|X^0|\times|X^1|\times\cdots\times|X^L|$ states,
the number of its transitions shown in \ref{item1_codet}) is bounded by $\sum_{i=0}^{L}
|T_{\ep}^i|$, where $T_{\ep}^i$ is the unobservable event set of automaton $\Scal_i$,
the number of its transitions shown in \ref{item2_codet}) is bounded by
$\prod_{i=0}^{L}|T_o^i|$, where $T_{o}^i$ is the observable event set of automaton $\Scal_i$,
$i\in[0,L]$. 
Hence $\CCa(\Scal_0;\Scal_1,\dots,\Scal_L)$ has at most 
$|X^0|^2\times|X^1|^2\times\cdots\times|X^L|^2\times(\prod_{i=0}^{L}|T_o^i|+\sum_{i=0}^{L}
|T_{\ep}^i|)$ transitions.

\section{Main results}\label{sec3:mainresult}

Consider an FSA $\Scal$ \eqref{DES_FSA}.
In order to formulate co-detectability, we first choose a number $L\in\Z_{+}$ of locations,
where in each 
location there is a {\it local observer} $\Osf_i$ who can observe a subset $T_i\subset T_o$ of
observable events via {\it local labeling function} $O_i:T_i\to\Sig\cup\{\ep\}$:
if an observable event $t\in T_i$ occurs,
then $\Osf_i$ observes $\ell(t)$, which is denoted by $O_i(t)=\ell(t)$; 
however, if an event $t\in T\setminus T_i$ occurs, 
then $\Osf_i$ observes nothing, which is denoted by $O_i(t)=\ep$.
In this sense, we also call labeling function $\ell$ the {\it global observer}.
$O_i$ is also recursively extended to $O_i:T^*\cup T^{\omega}\to\Sig^*\cup\Sig^{\omega}$.
Corresponding to observer $\Osf_i$, we say an event $t\in T_i$ is {\it $\Osf_i$-observable},
and $t\in T\setminus T_i$ is
{\it $\Osf_i$-unobservable}.
We write the set of local observers by $\Osf=\{\Osf_i|i\in[1,L]\}$.
For each $\s\in\Sig^*$, we denote the {\it current state estimate via $O_i$}
by $\Mt_{O_i}({\Scal},\s):=\{x\in X|(\exists x_0\in X_0)(\exists s\in T^*)[
(O_i(s)=\s)\wedge ((x_0,s,x)\in\dt)]\}$. Similarly, we denote
$\Mt_{O_i}(X',\s):=\{x\in X|(\exists x'\in X')(\exists s\in T^*)[
(O_i(s)=\s)\wedge ((x',s,x)\in\dt)]\}$ for all $X'\subset X$ and $\s\in\Sig^*$.
We also call automaton
\begin{equation}\label{local_DES_FSA}
	\Scal_i=(X,T,X_0,\dt,\Sig,O_i)
\end{equation}
{\it local automaton in location $i$} of automaton $\Scal$, $i\in[1,L]$.

\subsection{Co-detectability}

\subsubsection{Formulation}

We extend the notion of strong detectability studied in \cite{Shu2007Detectability_DES,Shu2011GDetectabilityDES}
(shown in Definition \ref{def_strdet}) to a
decentralized version (shown in Definition \ref{def_codet}).  
Definition \ref{def_strdet} implies that there is a time delay $k$, for each infinite-length 
event sequence generated by $\Scal$, every prefix of the label sequence generated by 
the event sequence of length no less than $k$ allows reconstructing the current state. 
More generally, Definition \ref{def_codet} implies that there is a time delay $k$,
for each infinite-length 
event sequence generated by $\Scal$, {\it in some location $i$},
every prefix of the label sequence generated by 
the event sequence {\it via $O_i$} of length no less than $k$ allows reconstructing the current state. 

\begin{definition}[StrDet]\label{def_strdet}
	An FSA $\Scal$ \eqref{DES_FSA} is called {\it strongly detectable} if
	\begin{align*}
		&(\exists k\in\N)(\forall s\in L^{\omega}({\Scal}))(\forall \s\sqsubset \ell(s))\\
		&[(|\s|\ge k)\implies(|\Mt({\Scal},\s)|=1)].
	\end{align*}
\end{definition}

\begin{definition}[CoDet]\label{def_codet}
	An FSA $\Scal$ \eqref{DES_FSA} is called {\it $\Osf$-co-detectable} if
	\begin{align*}
		&(\exists k\in\N)(\forall s\in L^{\omega}({\Scal}))(\exists i\in[1,L])(\forall \s\sqsubset O_i(s))\\
		&[(|\s|\ge k)\implies(|\Mt_{O_i}({\Scal},\s)|=1)].
	\end{align*}
\end{definition}

Definition \ref{def_codet}
means that an FSA $\Scal$ is $\Osf$-co-detectable if for each infinite-length event sequence
generated by $\Scal$, at least one local observer can determine the current and subsequent states
after a common observation time delay $k$ that does not depend on infinite-length event sequences.
When $L=1$ and $O_1=\ell$, Definition \ref{def_codet} reduces to Definition \ref{def_strdet}.
Note that when we consider $\Osf$-co-detectability of $\Scal$, 
the labeling function $\ell$ works if and only if $T_o=T_i$ for some $i\in[1,L]$, i.e., one local observer
can observe all observable events.
In this case, $\Osf$-co-detectability is equivalent to strong detectability.
Otherwise, one has $T_i\subsetneq T_o$ for all $i\in[1,L]$.

The notion of strong detectability studied in \cite{Zhang2019KDelayStrDetDES,Zhang2020bookDDS}
is strictly weaker than Definition \ref{def_strdet}. The former is defined as follows:
$(\exists k\in\N)(\forall \s\in\LM^{\omega}({\Scal}))(\forall \s'\sqsubset\s)[(|\s'|\ge k)\implies(|\Mt
({\Scal},\s')|=1)]$. The former can describe more systems than the latter, and they are quite close to
each other, so the former works better in a centralized setting.
However, since the former directly relies on observation sequences, it is not easy to
extend it to a decentralized setting, since in different locations, the local labeling functions
may differ.

The notion of strong co-detectability formulated in \cite{Shu2011CoDetDES} is as follows:
\begin{align*}
	&(\exists k\in\N)(\forall s\in L^{\omega}({\Scal}))(\forall \s\sqsubset\ell(s))\\
	&[(|\s|\ge k)\implies( (\exists i\in[1,L])[|\Mt_{O_i}({\Scal},O_i(s))|=1])].
\end{align*}
This notion is actually not well defined, because $\Mt_{O_i}({\Scal},O_i(s))$ may not be well 
defined, as $O_i(s)$ may be of infinite length. Even if after changing $O_i(s)$ to a prefix 
of itself, which makes this notion well defined, this notion is not very reasonable, because
the usage of $\ell(s)$ in this notion requires some local observer to observe all observable 
events, otherwise one may have $O_i(s)=\ep$. 
The notion actually implies 
strong detectability that is in a centralized setting.
We are not interested in extending the notion of weak detectability studied in 
\cite{Shu2007Detectability_DES,Shu2011GDetectabilityDES,Zhang2019KDelayStrDetDES,Zhang2020bookDDS}
to a decentralized version, since it is too weak so that it is very difficult to use this notion to 
determine the current and subsequent states.

The problem considered in this subsection is formulated as follows.

\begin{problem}[FSA CO-DETECTABILITY]\label{prob_codet}
	
	INSTANCE: An FSA $\Scal$ \eqref{DES_FSA} and a set $\Osf=\{\Osf_i|i\in[1,L]\}$ of local observers.

	QUESTION: Is $\Scal$ $\Osf$-co-detectable?
\end{problem}

The size of the input of Problem \ref{prob_codet} is $|\Scal|+\sum_{i=1}^{L}|T_i|\le
|\Scal|+L|T|$, where
$|\Scal|$ is the size of $\Scal$, i.e., the number of states and the number of transitions,
plus the labeling function, each $T_i$ is the set of $\Osf_i$-observable events.

\subsubsection{Equivalent condition}

In the sequel, in order to develop an equivalent condition for $\Osf$-co-detectability 
without any assumption or changing the FSAs, we consider its negation. To implement 
this idea, we first show negation of $\Osf$-co-detectability as follows.

\begin{proposition}\label{prop1_codet}
	An FSA $\Scal$ \eqref{DES_FSA} is not $\Osf$-co-detectable if and only if
	\begin{align*}
		&(\forall k\in\N)(\exists s_k\in L^{\omega}({\Scal}))(\forall i\in[1,L])
		(\exists \s_i\sqsubset O_i(s_k))\\
		&[(|\s_i|\ge k)\wedge(|\Mt_{O_i}({\Scal},\s_i)|>1)].
	\end{align*}
\end{proposition}

In order to derive an equivalent condition for negation of $\Osf$-co-detectability,
we need an extra structure. Given a concurrent composition $\CCa(\Scal;\Scal_1,\dots,\Scal_L)$ 
\eqref{ConComp_automata},
we compute a variant
\begin{equation}\label{ConCompVariant_automata}
\CCadia(\Scal;\Scal_1,\dots,\Scal_L)
\end{equation}
from \eqref{ConComp_automata} as follows:
(\romannumeral1) Add new states to \eqref{ConComp_automata} to make its state set become
$X\times(X\cup\{\diamond\})^{L}$,
where $\diamond$ is a fresh state not in $X$. (\romannumeral2)
For each state $(x_0,x_1,\dots,x_L)$ in $X\times(X\cup\{\diamond\})^{L}$ satisfying $x_0\ne x_i$
for some $i\in[1,L]$,
at $x_0$ we choose an arbitrary transition $(x_0,t_0,x_0')\in\dt$ with $t_0\in T\cup\{\ep\}$;
at each $x_j$, $j\in[1, L]$, such that  
$x_j=x_0$, choose an arbitrary transition $(x_j,t_j,x_j')\in\dt$ with $t_j\in T\cup\{\ep\}$;
at each $x_k$, $k\in[1, L]$, such that 
$x_k\ne x_0$, we define transition $x_k\xrightarrow[]{t_k}x_k'$, where $t_k=x_k'=\diamond$; add transition 
$(x_0,x_1,\dots,x_L)\xrightarrow[]{(t_0,t_1,\dots,t_L)}(x_0',x_1',\dots,x_L')$ to \eqref{ConComp_automata}
whenever $(x_0,x_{j_1},\dots,x_{j_l})\xrightarrow[]{(t_0,t_{j_1},\dots,t_{j_l})}(x_0',x_{j_1}',\dots,
x_{j_l}')$ is a transition of $\CCa(\Scal;\Scal_{j_1},\dots,\Scal_{j_l})$ or $t_0=t_{j_1}=\cdots=t_{j_l}=\ep$,
where $x_{j_1},\dots,x_{j_l}$ are exactly all states in $\{x_1,\dots,x_L\}$ such that $x_{j_1}=\cdots=x_{j_l}=
x_0$.
See Fig. \ref{fig2:codet} for a sketch. 
(\romannumeral3) 
We define $\ell(\diamond)=\diamond$. 
We have obtained \eqref{ConCompVariant_automata}.

Intuitively, in \eqref{ConCompVariant_automata}, for every transition from state $\bar x'$ to state $\hat x'$,
the number of $\diamond$'s in $\hat x'$ is no less than the number of $\diamond$'s in $\bar x'$, and in addition,
once a component in $\bar x'$ equals $\diamond$, then the same component in $\hat x'$ must also be $\diamond$. This observation can be stated as follows.
\begin{proposition}\label{prop2_co_DES}
	In \eqref{ConCompVariant_automata}, there is a transition sequence $(x_0,x_1,\dots,x_L)\xrightarrow[]
	{s'}(x_0',\diamond,\dots,\diamond)$ with $x_1,\dots,x_L\in X$
	if and only if for each $i\in[1,L]$, in $\CCa(\Scal;\Scal_i)$, there is a transition sequence 
	$(x_0,x_i)\xrightarrow[]{s}(\bar x_0,\bar x_i)$ such that $\bar x_0\ne\bar x_i$ and the 
	$0$-th component 
	of $(x_0,x_i)\xrightarrow[]{s}(\bar x_0,\bar x_i)$
	is a prefix of the $0$-th component of $(x_0,x_1,\dots,x_L)\xrightarrow[]
	{s'}(x_0',\diamond,\dots,\diamond)$.
\end{proposition}


\begin{figure}[htpb!]
	\begin{center}
		\begin{align*}
		\begin{array}[]{rcl}
			x_0\\
			x_0\\
			x_2\\
			\diamond
		\end{array}
		\quad
		\implies\quad
		\begin{array}[]{lll}
			x_0&\xrightarrow[]{t_0}&x_0'\\
			x_0&\xrightarrow[]{t_1}&x_1'\\
			x_2&\xrightarrow[]{\diamond}&\diamond\\
			\diamond&\xrightarrow[]{\diamond}&\diamond
		\end{array}
		\end{align*}
	\end{center}
	\caption{Sketch for computing transitions of $\CCadia(\Scal;\Scal_1,\dots,\Scal_L)$ \eqref{ConCompVariant_automata}.}
	\label{fig2:codet}
\end{figure}

%
%

\begin{example}\label{exam1_codet}
	Consider FSA $\Scal$ shown in Fig. \ref{fig1:codet}, where the labeling function $\ell$ is the
	identity map, i.e., all events $a,b,c,d$ can be directly observed. Consider two local observers
	$\Osf_1$ and $\Osf_2$, where $a,b$ can be observed by both observers, but $c$ can only be
	observed by $\Osf_1$, $d$ can only be observed by $\Osf_2$. That is, for $\Osf_1$, one has
	$O_1(a)=a$, $O_1(b)=b$, $O_1(c)=c$, and $O_1(d)=\ep$; for $\Osf_2$, one has
	$O_2(a)=a$, $O_2(b)=b$, $O_2(c)=\ep$, and $O_2(d)=d$.
	The local automaton corresponding to
	observer $\Osf_i$ is denoted by $\Scal_i$, $i=1,2$. Part of the concurrent composition 
	$\CCadia(\Scal;\Scal_1,\Scal_2)$ defined by \eqref{ConCompVariant_automata}
	is drawn in Fig. \ref{fig3:codet}.

	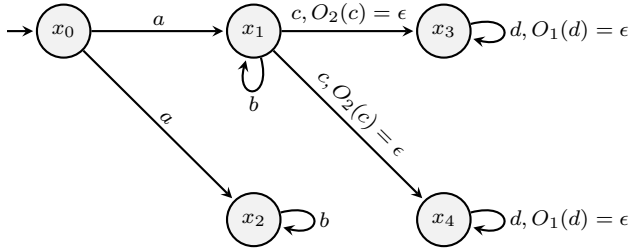
\begin{figure}[htpb!]
			\tikzset{global scale/.style={
    scale=#1,
    every node/.append style={scale=#1}}}
		\begin{center}
			\begin{tikzpicture}[global scale = 1.0,
				>=stealth',shorten >=1pt,thick,auto,node distance=2.5 cm, scale = 1.0, transform shape,
	->,>=stealth,inner sep=2pt,
				every transition/.style={draw=red,fill=red,minimum width=1mm,minimum height=3.5mm},
				every place/.style={draw=blue,fill=blue!20,minimum size=7mm}]
				\tikzstyle{emptynode}=[inner sep=0,outer sep=0]
				\node[state, initial, initial where = left] (x0) {$x_0$};
				\node[state] (x1) [right of = x0] {$x_1$};
				\node[state] (x2) [below of = x1] {$x_2$};
				\node[state] (x3) [right of = x1] {$x_3$};
				\node[state] (x4) [below of = x3] {$x_4$};

				\path[->]
				(x0) edge node [above, sloped] {$a$} (x1)
				(x0) edge node [above, sloped] {$a$} (x2)
				(x1) edge node [above, sloped] {$c,O_2(c)=\ep$} (x3)
				(x1) edge node [above, sloped] {$c,O_2(c)=\ep$} (x4)
				(x1) edge [loop below] node [below, sloped] {$b$} (x1)
				(x3) edge [loop right] node {$d,O_1(d)=\ep$} (x3)
				(x2) edge [loop right] node {$b$} (x2)
				(x4) edge [loop right] node {$d,O_1(d)=\ep$} (x4)
				;

			\end{tikzpicture}
	\end{center}
	\caption{FSA $\Scal$ with two local automata $\Osf_1$ and $\Osf_2$,
	where a state with an input arrow from nowhere denotes an initial one, e.g., $x_0$.}
	\label{fig1:codet}
	\end{figure}

	\begin{figure}[htpb!]
			\tikzset{global scale/.style={
    scale=#1,
    every node/.append style={scale=#1}}}
		\begin{center}
			\begin{tikzpicture}[global scale = 1.0,
				>=stealth',shorten >=1pt,thick,auto,node distance=2.5 cm, scale = 1.0, transform shape,
	->,>=stealth,inner sep=2pt,
				every transition/.style={draw=red,fill=red,minimum width=1mm,minimum height=3.5mm},
				every place/.style={draw=blue,fill=blue!20,minimum size=7mm}]
				\tikzstyle{emptynode}=[inner sep=0,outer sep=0]
				\tikzstyle{stateCom}=[shape=rectangle, draw, thick, fill=gray!10]
				\node[stateCom, initial, initial where = left] (000) {$\begin{matrix}x_0\\ x_0\\ x_0\end{matrix}$};
				\node[stateCom] [right of = 000] (111) {$\begin{matrix}x_1\\ x_1\\ x_1\end{matrix}$};
				\node[stateCom] [right of = 111] (344) {$\begin{matrix}x_3\\ x_4\\ x_1\end{matrix}$};
				\node[stateCom] [right of = 344] (3dd) {$\begin{matrix}x_3\\ \diamond\\ \diamond\end{matrix}$};
				\node[stateCom] [above of = 344] (113) {$\begin{matrix}x_1\\ x_1\\ x_3\end{matrix}$};
				\node[stateCom] [right of = 113] (11d) {$\begin{matrix}x_1\\ x_1\\ \diamond\end{matrix}$};

				\path[->]
				(000) edge node {$(a,a,a)$} (111)
				(111) edge [loop above] node [above, sloped] {$(b,b,b)$} (111)
				(111) edge node {$(c,c,\ep)$} (344)
				(344) edge node {$(d,\diamond,\diamond)$} (3dd)
				(111) edge node [above, sloped] {$(\ep,\ep,c)$} (113)
				(113) edge node [above, sloped] {$(b,b,\diamond)$} (11d)
				;

			\end{tikzpicture}
	\end{center}
	\caption{Part of concurrent composition $\CCadia(\Scal;\Scal_1,\Scal_2)$, where $\Scal$
	and its two local automata $\Scal_1,\Scal_2$ are shown in Fig. \ref{fig1:codet} and in Example 
	\ref{exam1_codet}.}
	\label{fig3:codet}
	\end{figure}
	
\end{example}

\begin{theorem}\label{thm1_codet}
	An FSA $\Scal$ \eqref{DES_FSA} is not $\Osf$-co-detectable if and only if
	in 
	concurrent composition 
	$\CCadia(\Scal;\Scal_1,\dots,\Scal_L)$ \eqref{ConCompVariant_automata},
	where each $\Scal_i$ is the local automaton corresponding to observer
	$\Osf_i$ defined by \eqref{local_DES_FSA}, $i\in[1,L]$,  
	\begin{subequations}\label{eqn1_codet}
		\begin{align}
			&\text{there is a transition sequence}\\
			&x_0'\xrightarrow[]{s_1'}x_1'\xrightarrow[]{s_2'}x_{1}'
			\xrightarrow[]{s_3'}x_2'\xrightarrow[]{s_4'}x_{2}'\xrightarrow[]{s_5'}\cdots
			\xrightarrow[]{s_{2L-1}'}x_L'\xrightarrow[]{s_{2L}'}x_{L}'\xrightarrow[]{s_{2L+1}'}x_{L+1}'\\
			&\text{such that }x_0'\text{ is initial;}\label{eqn1c_codet}\\
			&\text{for each }i\in[1,L],\text{ there is }k_i\in[1,L]\text{ such that }x_i'(k_i)\in X,\nonumber\\
			&O_{k_i}(s_{2i}'(0))=O_{k_i}(s_{2i}'(i))\in\Sig^+,\text{ and }k_j\ne k_{j'}\text{ for all different }j,j'\in[1,L];\label{eqn1d_codet}\\
			&\text{for all }i\in[1,L],\text{ }x_{L+1}'(i)=\diamond;\label{eqn1e_codet}\\
			&\text{and in }\Scal,\text{ there is a transition cycle reachable from }x_{L+1}'(0).\label{eqn1f_codet}
		\end{align}
	\end{subequations}
\end{theorem}

\begin{proof}
	``if'': Assume Eqn. \eqref{eqn1_codet} holds. According to \eqref{eqn1_codet}, for all $k\in\Z_{+}$, we choose 
	$$s_1'(0)(s_2'(0))^{k}s_3'(0)(s_4'(0))^k\dots s'_{2L-1}(0)(s'_{2L}(0))^ks_{2L+1}(0)=:\hat s'\in
	L(\Scal),$$ then we have for every $i\in[1,L]$, 
	$$|O_{k_i}(s_1'(0)(s_2'(0))^k\dots s_{2i-1}'(0)(s_{2i}'(0))^k)|\ge ik,$$
	and there is $\s_i\sqsubset O_{k_i}(\hat s')$
	such that $$|\Mt_{O_{k_i}}(\Scal,\s_i)|>1.$$
	In addition, the transition cycle reachable from
	$x_{2L+1}'(0)$ can be extended to an infinite transition sequence, we conclude that $\Scal$ is not 
	$\Osf$-co-detectable by Proposition \ref{prop1_codet}.

	``only if'': Assume $\Scal$ is not $\Osf$-co-detectable. Choose sufficiently large $n\in\N$, then
	by Propositions \ref{prop1_co_DES}, \ref{prop1_codet}, and \ref{prop2_co_DES}, 
	in \eqref{ConCompVariant_automata}, there is a transition sequence
	\begin{equation}\label{eqn2_codet}
		x_0'\xrightarrow[]{s'}x_{L+1}'
	\end{equation}
	such that
	\begin{subequations}\label{eqn3_codet}
		\begin{align}
			&x_0'\text{ is initial};\\
			&x_{L+1}'(i)=\diamond\text{ for all }i\in[1,L];\\
			&\text{for every }i\in[1,L],\text{ there is }s_i'\sqsubset s'\text{ such that }
			|O_i(s_i'(0))|\ge n,\\
			&\text{and }x_i'(0)\ne x_i'(i)\in X,\text{ where }x_i'\text{ is such that }x_0'\xrightarrow[]{s_i'}x_i'
			\text{ is a prefix of }\eqref{eqn2_codet};\\
			&\text{and in }\Scal,\text{ there is a transition cycle reachable from }x_{L+1}'(0).
		\end{align}
	\end{subequations}
	
	Fix $j\in[1,L]$, and consider the above sub-transition sequence 
	\begin{equation}\label{eqn4_codet}
		x_0'\xrightarrow[]{s_j'}x_j'
	\end{equation}
	of \eqref{eqn2_codet}. Since $|O_j(s_j'(0))|\ge n$, and $n$ is sufficiently large,
	by the Pigeonhole Principle and the structure of \eqref{ConCompVariant_automata},
	in \eqref{eqn4_codet} there is a sub-transition cycle
	\begin{equation}\label{eqn8_codet}
		\bar x_j'\xrightarrow[]{\bar s_j'}\bar x_j' 
	\end{equation}
	such that $\bar x_j'(j)\in X$ and $O_j(\bar s_j'(0))\in\Sig^+$.
	We can make the following two modifications repetitively to \eqref{eqn2_codet}
	in order to obtain an extension 
	of \eqref{eqn2_codet} such that for all different $i,j\in[1,L]$, there exist non-overlap 
	sub-transition cycles shown in \eqref{eqn8_codet}, and \eqref{eqn3_codet} is still satisfied. 
	Thus, \eqref{eqn1_codet} holds.

	(\romannumeral1) 
	Assume for some different $i,j\in[1,L]$, there is a sub-transition sequence 
	\begin{equation}\label{eqn9_codet}
		x_i'\xrightarrow[]{s_{i_1}'}x_j'\xrightarrow[]{s_{j_1}'}x_i'\xrightarrow[]{s_{i_2}'}x_j'
	\end{equation}
	in \eqref{eqn2_codet} such that $O_i(s_{i_1}'(0)s_{j_1}'(0))\in\Sig^+$,
	$O_j(s_{j_1}'(0)s_{i_2}'(0))\in\Sig^+$, and $x'_i(i),x'_j(j)\in X$.
	We replace \eqref{eqn9_codet} by
	\begin{equation}\label{eqn10_codet}
		x_i'\xrightarrow[]{s_{i_1}'}x_j'\xrightarrow[]{s_{j_1}'}x_i'\xrightarrow[]{s_{i_2}'}x_j'
		\xrightarrow[]{s_{j_1}'}x_i'\xrightarrow[]{s_{i_2}'}x_j',
	\end{equation}
	where \eqref{eqn10_codet} contains non-overlap sub-transition cycles
	$x_i'\xrightarrow[]{s_{i_1}'}x_j'\xrightarrow[]{s_{j_1}'}x_i'$ and 
	$x_j'\xrightarrow[]{s_{j_1}'}x_i'\xrightarrow[]{s_{i_2}'}x_j'$.
 
	(\romannumeral2)
	On the other hand, assume for some different $i,j\in[1,L]$, there is a sub-transition sequence 
	\begin{equation}\label{eqn11_codet}
		x_i'\xrightarrow[]{s_{i_1}'}x_j'\xrightarrow[]{s_{j_1}'}x_j'\xrightarrow[]{s_{i_2}'}x_i'
	\end{equation}
	in \eqref{eqn2_codet} such that $O_i(s_{i_1}'(0)s_{j_1}'(0)s_{i_2}'(0))\in\Sig^+$,
	$O_j(s_{j_1}'(0))\in\Sig^+$, and $x'_i(i),x'_j(j)\in X$.
	We replace \eqref{eqn11_codet} by
	\begin{equation}\label{eqn12_codet}
		x_i'\xrightarrow[]{s_{i_1}'}x_j'\xrightarrow[]{s_{j_1}'}x_j'\xrightarrow[]{s_{i_2}'}x_i'
		\xrightarrow[]{s_{i_1}'}x_j'\xrightarrow[]{s_{j_1}'}x_j'\xrightarrow[]{s_{i_2}'}x_i',
	\end{equation}
	where \eqref{eqn12_codet} contains non-overlap sub-transition cycles
	$x_i'\xrightarrow[]{s_{i_1}'}x_j'\xrightarrow[]{s_{j_1}'}x_j'\xrightarrow[]{s_{i_2}'}x_i'$
	$x_j'\xrightarrow[]{s_{j_1}'}x_j'$.
\end{proof}

\begin{example}\label{exam2_codet}
	Reconsider the FSA $\Scal$ in Example \ref{exam1_codet} (in Fig. \ref{fig1:codet})
	with its two local automata 
	$\Scal_1$ and $\Scal_2$. We verify its $\{\Osf_1,\Osf_2\}$-co-detectability by Theorem \ref{thm1_codet}.
	In the concurrent composition $\CCadia(\Scal;\Scal_1,\Scal_2)$, there is a transition sequence
	\begin{align*}
		&(x_0,x_0,x_0)\xrightarrow[]{(a,a,a)}(x_1,x_1,x_1)\xrightarrow[]{(b,b,b)}(x_1,x_1,x_1)
		\xrightarrow[]{(c,c,\ep)}\\
		&(x_3,x_4,x_1)\xrightarrow[]{(d,\diamond,\diamond)}
		(x_3,\diamond,\diamond)
	\end{align*} shown 
	in Fig. \ref{fig3:codet} such that $(x_0,x_0,x_0)$ is initial, satisfying \eqref{eqn1c_codet};
	in self-loop $(x_1,x_1,x_1)\\\xrightarrow[]{(b,b,b)}(x_1,x_1,x_1)$, all components of
	$(x_1,x_1,x_1)$ are states of $\Scal$, and $O_1(b)=O_2(b)=b$ are of positive length, satisfying
	\eqref{eqn1d_codet}; the latter two components of state $(x_3,\diamond,\diamond)$ are both 
	$\diamond$, satisfying \eqref{eqn1d_codet}; in $\Scal$, there is a transition cycle $x_3
	\xrightarrow[]{d}x_3$ reachable from the $0$-th component $x_3$ of $(x_3,\diamond,\diamond)$,
	satisfying \eqref{eqn1e_codet}. Hence $\Scal$ is not $\{\Osf_1,\Osf_2\}$-co-detectable.
\end{example}

\subsubsection{Complexity analysis}

Next we prove that Theorem \ref{thm1_codet} provides a $\PSPACE$ verification algorithm
for $\Osf$-co-detectability.

\begin{theorem}\label{thm2_codet}
	Problem \ref{prob_codet} belongs to $\PSPACE$.
\end{theorem}

\begin{proof}
	By Theorem \ref{thm1_codet}, checking negation of $\Osf$-co-detectability is equivalent to 
	checking whether \eqref{eqn1_codet} holds.
	
	Consider $\CCadia(\Scal;\Scal_1,\dots,\Scal_L)$ \eqref{ConCompVariant_automata}, but we
	do not construct it explicitly.
	Guess integers $k_1,\dots,k_L$ between $1$ and $L$ such that
	they are pairwise different, guess states $x_1',x_2',\dots,x_{L+1}'$ of 
	\eqref{ConCompVariant_automata}. 
	Check (\romannumeral1) $x_1'$ is reachable,
	(\romannumeral2) for all $i\in[1,L]$, $x_{i+1}'$ equals $x_{i}'$ or $x_{i+1}'$ is reachable
	from $x_{i}'$, (\romannumeral3) for all $i\in[1,L]$, $x_i'(k_i)\in X$,  (\romannumeral4)
	for all $i\in[1,L]$, there is a transition cycle $x_i'\xrightarrow
	[]{s_i'}x_i'$ such that $O_{k_i}(s_i'(0))\in\Sig^+$, (\romannumeral5) $x_{L+1}(l)=\diamond$ for all 
	$l\in[1,L]$, and (\romannumeral6) there is a transition cycle
	in $\Scal$ reachable from $x_{L+1}'(0)$. All above checking could be done by nondeterministic 
	search based on the transition relation of $\CCadia(\Scal;\Scal_1,\dots,\Scal_L)$.
	Hence whether \eqref{eqn1_codet} holds could be verified in $\NPSPACE$, i.e., in $\PSPACE$
	by Savitch's theorem \cite{Savitch1970PspaceNPspace}.
	Then by $\PSPACE=\coPSPACE$, Problem \ref{prob_codet} belongs to $\PSPACE$.
\end{proof}

\begin{corollary}
	Strong detectability of FSA \eqref{DES_FSA} can be verified in $\PTIME$.
\end{corollary}

\begin{proof}
	The condition in Theorem \ref{thm1_codet} in case of $L=1$ and $O_1=\ell$ can be verified in 
	linear time in the size of $\CCadia(\Scal;\Scal)$ by computing all its strongly connected 
	components, and $\CCadia(\Scal;\Scal)$ can be computed in quadratic polynomial time in
	the size of $\Scal$.
\end{proof}

In \cite{Masopust2018ComplexityDetectabilityDES},
the $\NL$-hardness of verifying strong detectability of deterministic and deadlock-free
FSAs with a single initial state and a single (observable) event was proved, an $\NL$ upper 
bound for verifying
strong detectability of FSAs was also given under Assumption \ref{assum1_Det_PN}.
Next we show an $\NL$ upper bound without any assumption.

\begin{theorem}\label{thm4_codet}
	The problem of verifying strong detectability of FSA \eqref{DES_FSA} belongs to $\NL$.
\end{theorem}

\begin{proof}
	Guess states $x_1,\bar x_1,x_2,x_3$ of $X$.
	Check in $\CCadia(\Scal;\Scal)$: (\romannumeral1) $(x_1,\bar x_1)$ is reachable,
	(\romannumeral2) $(x_1,\bar x_1)$ belongs to an observable transition cycle, 
	(\romannumeral3) $(x_2,\diamond)$ is reachable from $(x_1,\bar x_1)$; and in $\Scal$:
	(\romannumeral4) $x_3$ is equal to $x_2$ or $x_3$ is reachable from $x_2$,
	(\romannumeral5) $x_3$ belongs to a transition cycle, all by nondeterministic search.
\end{proof}

In order to give a lower bound for negation of co-detectability, we adopt the 
$\NP$-complete acyclic DFA INTERSECTION problem
shown in \cite[Theorem 1]{Rampersad2010NP-hardAcycDFA}.
\begin{problem}[DFA INTERSECTION]\label{prob_DFAIntersec}

	INSTANCE: DFAs $\A_1,\dots,\A_n$ over the same alphabet $\Sig$.

	QUESTION: Does there exist $\s\in\Sig^*$ such that $\s$ is accepted by each
	$\A_i$, $1\le i\le n$?
\end{problem}

\begin{proposition}[\cite{Rampersad2010NP-hardAcycDFA}]\label{prop2_codet}
	Problem \ref{prob_DFAIntersec} is $\NP$-complete if the input DFAs are all acyclic.
\end{proposition}

Note that the original version proved in \cite[Theorem 1]{Rampersad2010NP-hardAcycDFA}
is such that Problem \ref{prob_DFAIntersec} is $\NP$-complete if the input DFAs all
recognize finite languages.
However, in the proof of \cite[Theorem 1]{Rampersad2010NP-hardAcycDFA}, the well-known
$\NP$-complete $3$-$\SAT$ problem is reduced in polynomial time to Problem \ref{prob_DFAIntersec}
when all input DFAs are acyclic. Acyclic DFAs all recognize finite languages.
Hence Proposition \ref{prop2_codet} slightly strengthens \cite[Theorem 1]{Rampersad2010NP-hardAcycDFA}.

In order to analyze the complexity of co-detectability of FSAs, we strengthen Proposition
\ref{prop2_codet} slightly as follows.

\begin{proposition}\label{prop3_codet}
	Problem \ref{prob_DFAIntersec} is $\NP$-complete if the input 
	DFAs $\A_1,\dots,\A_n$ are acyclic, $\A_1$ has exactly one accepting state,
	and all other $\A_i$'s have all states accepting.
\end{proposition}

\begin{proof}
	We are given acyclic DFAs $\A_1,\dots,A_n$ over the same alphabet $\Sig$.
	Construct acyclic DFA $\A'_1$ from $\A_1$ by adding
	transitions $q\xrightarrow[]{\lambda}\diamond_1$
	at each accepting state $q$, changing all accepting states to be
	non-accepting, and changing $\diamond_1$ to be accepting.
	For each $2\le i\le n$, construct acyclic DFA $\A'_i$ from $\A_i$ by adding
	transitions $q\xrightarrow[]{\lambda}\diamond_i$
	at each accepting state $q$, changing all non-accepting states, including $\diamond_i$,
	to be accepting. Then $\A_1'$ has exactly one accepting state $\diamond_1$, and $\A_2',\dots,\A_n'$
	have all states accepting. One sees that for each word $w\in\Sig^*$, $w$ is accepted 
	by all $\A_i$'s if and only if $w\lambda$ is accepted by all $\A_1',\dots,\A_n'$.
	By Proposition \ref{prop2_codet}, this proposition holds.
\end{proof}

Next we give a lower bound for co-detectability, where the reduction is inspired from 
the $\PSPACE$-hardness proof of verifying co-diagnosability of deterministic FSAs
\cite{Cassez2012ComplexityCodiagnosability}, instead, acyclic DFAs (but not general DFAs)
are chosen to do the reduction according to features of co-detectability.

\begin{theorem}\label{thm3_codet}
	Problem \ref{prob_codet} is $\coNP$-hard for deterministic FSAs.
\end{theorem}

\begin{proof}
	We prove this result by Proposition \ref{prop3_codet}.
	
	We are given acyclic DFAs $\A_0,\dots,\A_n$ over the same alphabet $\Sig$
	such that $\A_0$ has exactly one accepting state and all other $\A_i$'s have 
	all states accepting. Next we construct an FSA $\Scal$ from $\A_0,\dots,\A_n$ in polynomial time
	as shown in Fig. \ref{fig4:codet}.
	In each $\A_i$, $i\in[0,L]$, change each letter (i.e., event) $\s\in\Sig$ to $\s_i$,
	and set $\Sig_i=\{\s_i|\s\in\Sig\}$ and $\ell(\s_i)=\s$, where $\ell$ is the labeling function. 
	Add initial state $\diamond_0$ and transition $\diamond_0\xrightarrow[]{a_i}q$ to the initial state $q$
	of each $\A_i$, $i\in[0,L]$, where $\diamond_0$ differs from any state of any $\A_i$,
	$a_i\notin\bigcup_{i\in[0,L]} \Sig_i$, and set $\ell(a_i)=a$. 
	Change each initial state of each $\A_i$ to be non-initial,
	$i\in[0,L]$. Add two states $\diamond_1$ and $\diamond_2$ that are not any state of any $\A_i$,
	$i\in[0,L]$,
	and self-loops on them labeled by event $a$, and set $\ell(a)=a$.
	At the accepting state of $\A_0$, add transition to $\diamond_1$ labeled by event $a$.
	At each non-accepting state of $\A_0$ and each state of each $\A_i$, $i\in[1,L]$,
	add transition to $\diamond_2$ labeled by $a$. We have obtained the deterministic FSA $\Scal$
	with a unique initial state.
	The number of states of $\Scal$ equals the sum of numbers of states of all $\A_i$, $i\in[0,L]$,
	plus $3$ (corresponding to newly added states $\diamond_1,\diamond_2,\diamond_3$).
	The event set of $\Scal$ is $\bigcup_{i\in[0,L]}\Sig_i\cup\{a\}\cup\{a_i|i\in[0,L]\}$.

	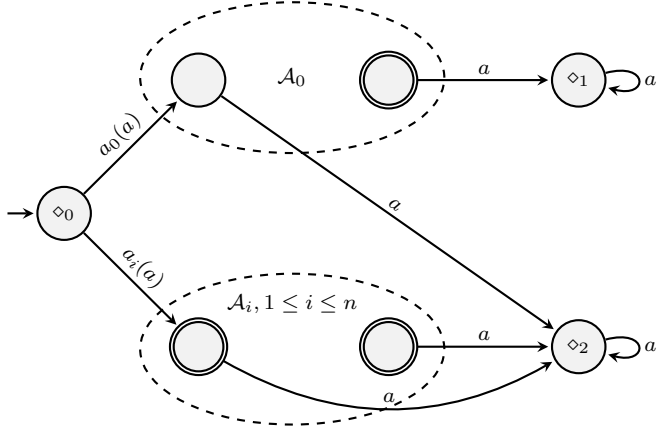
\begin{figure}[htbp]
		\tikzset{global scale/.style={
    scale=#1,
    every node/.append style={scale=#1}}}
		\begin{center}
			\begin{tikzpicture}[global scale = 1.0,
				>=stealth',shorten >=1pt,thick,auto,node distance=2.5 cm, scale = 1.0, transform shape,
	->,>=stealth,inner sep=2pt,
				every transition/.style={draw=red,fill=red,minimum width=1mm,minimum height=3.5mm},
				every place/.style={draw=blue,fill=blue!20,minimum size=7mm}]
				\tikzstyle{emptynode}=[inner sep=0,outer sep=0]
				\node[state, initial] (dia0) {$\diamond_0$};
				\node[state] (s00) [above right of = dia0] {};
				\node[state, accepting] (s0f) [right of = s00] {};
				\node[state, accepting] (sif) [below right of = dia0] {};
				\node[state, accepting] (sif') [right of = sif] {};
				\node[state] (dia1) [right of = s0f] {$\diamond_1$};
				\node[state] (dia2) [right of = sif'] {$\diamond_2$};

				\path[->]
				(dia0) edge node [above, sloped] {$a_0(a)$} (s00)
				(dia0) edge node [above, sloped] {$a_i(a)$} (sif)
				(s0f) edge node [above, sloped] {$a$} (dia1)
				(s00) edge node [above, sloped] {$a$} (dia2)
				(sif) edge [bend right] node [above, sloped] {$a$} (dia2)
				(sif') edge node [above, sloped] {$a$} (dia2)
				(dia1) edge [loop right] node {$a$} (dia1)
				(dia2) edge [loop right] node {$a$} (dia2)
				;

				\draw[dashed] (3.0,1.8) ellipse (2cm and 1cm);
				\node at (3.0,1.8) {$\A_0$};

				\draw[dashed] (3.0,-1.8) ellipse (2cm and 1cm);
				\node at (3.0,-1.2) {$\A_i,1\le i\le n$};

			\end{tikzpicture}
			\caption{Sketch of the reduction in the proof of Theorem \ref{thm3_codet}.}
			\label{fig4:codet}
		\end{center}
	\end{figure}

	Now we specify local automaton $\Scal_i$, $i\in[1,L]$, we only need to specify the labeling function $O_i$
	of each $\Scal_i$. For each $i\in[1,L]$, the observable event set is
	$\Sig_0\cup\Sig_i\cup\{a\}\cup\{a_i|i\in[0,L]\}$.
	That is, at each location $i\in[1,L]$, observer $\Osf_i$ can observe $\A_0$, $\A_i$, and all
	transitions outside $\bigcup_{j\in[0,L]}\A_j$.

	In order to prove this theorem, we only need to prove that
	there is a word $w\in\Sig^*$ that is accepted by all $\A_i$, $i\in[0,L]$, if and only if,
	$\Scal$ is not $\Osf$-co-detectable.

	$\Rightarrow$: Assume $w=w^1\dots w^n\in\Sig^*$ that is accepted by all $\A_i$, $i\in[0,L]$,
	where $w^1,\dots,w^n\in\Sig$, $n\in\N$.
	Then $a_iw^1_i\dots w^n_ia^{\omega}\in L^{\omega}(\Scal)$ for all $i\in[0,L]$.
	For each $i\in[1,L]$, $O_i(a_iw^1_i\dots w^n_ia^{\omega})=awa^{\omega}$.
	At each location $i\in[1,L]$, for each $n\in\Z_{+}$, the current state estimate via $O_i$ is
	$\Mt_{O_i}(\Scal,awa^n)=\{\diamond_1,\diamond_2\}$. That is, $\Scal$ is not $\Osf$-co-detectable
	by Proposition \ref{prop1_codet}.

	$\Leftarrow$: Since each DFA $\A_i$, $i\in[0,L]$,
	is acyclic, it can only generate at most finitely many words (i.e., finite event sequences),
	denote the length of the longest
	words generated by $\A_i$ by $k_i$. And by the structure of $\Scal$,
	each possible infinite event sequence generated by $\Scal$ can only be of the form $a_iwa^{\omega}$
	for some $i\in[0,L]$, where $w\in\Sig_i^*$, and $|w|\le k_i$.
	Assume that $\Scal$ is not $\Osf$-co-detectable. Choose $m=2+\max\{k_i|i\in[0,L]\}$, 
	one has there is $a_\iota ua^{\omega}\in L^{\omega}(\Scal)$ for some $\iota\in[0,L]$ 
	and $u\in\Sig_{\iota}^*$ such that 
	$\Mt_{O_j}(\Scal,a O_i(u)a^{n})=\{\diamond_1,\diamond_2\}$ in each location $i\in[1,L]$,
	where $|a O_i(u)a^n|\ge m$. Assume $u=\ep$. Then one has $\ep$ leads $\A_0$ to an accepting state
	and further to $\diamond_1$, hence $\A_0$ has only one state, and $\ep$ is accepted by $\A_0$.
	Note that $\ep$ is always accepted by all 
	other $\A_i$, $i\in[1,L]$, since all states of $\A_i$ are accepting. Assume $u\ne\ep$.
	Assume $\iota=0$. Then for each $i\in[1,L]$, $O_i(a_{\iota}ua^n)=a\ell(u)a^n$. Since $\ell(u)$ leads $\A_0$
	to an accepting state and further to $\diamond_1$, 
	word $\ell(u)$ is accepted by $\A_0$. $\ell(u)$ is also accepted by all $\A_i$, $i\in[1,L]$,
	by negation of $\Osf$-co-detectability.
	Assume $\iota\in[1,L]$. If $L=1$, then $\ell(u)=O_\iota(u)$ is accepted by $\A_0$
	and $\A_{\iota}$. If $L>1$, choose $j\in[1,L]$, $j\ne\iota$, local observer $\Osf_j$ observes 
	$aa^n$, then by $\Mt_{O_j}(\Scal,aa^n)=\{\diamond_1,\diamond_2\}$, one has $\ep$ is accepted by
	$\A_0$. $\ep$ is always accepted by all $\A_i$, $i\in[1,L]$. Based on the above discussion,
	one always has there is a word
	$u\in\Sig^*$ that is accepted by all $\A_i$, $i\in[0,L]$, which completes the proof.
\end{proof}

\subsection{Co-diagnosability}

In this subsection, we study co-diagnosability. 
To characterize do-diagnosability, we need to simplify concurrent composition 
\eqref{ConComp_automata}, but not to extend it as in dealing with co-detectability,
because in co-diagnosability, we only need to count occurrences of events, but in 
co-detectability, we need to count generated outputs.

\subsubsection{Formulation}

We specify a subset $T_f\subset T$ of faulty
events. Transitions containing a faulty event are called {\it faulty transitions}. In the literature, e.g., in
\cite{Sampath1995DiagnosabilityDES,Jiang2001PolyAlgorithmDiagnosabilityDES,Yoo2002DiagnosabiliyDESPTime},
it is widely assumed that all faulty events are unobservable without loss of generality,
since if they are observable, then their
occurrences could be directly observed. However, we consider a more general case
such that there may exist different observable events that produce the same label. So, 
not all occurrences of all observable faulty events can be directly observed here.
For $s\in L(\Scal)$, we 
write $T_f\in s$ if $s$ contains a faulty event, and $T_f\notin s$ otherwise as usual.

\begin{definition}[CoDiag]\label{def_codiag}
	An FSA $\Scal$ \eqref{DES_FSA} is called {\it $(\Osf,T_f)$-co-diagnosable} if
	\begin{align*}
		&(\exists k\in\N)(\forall s\in L(\Scal)\cap T^*T_f)(\forall s':ss'\in L(\Scal))\\
		&[(|s'|\ge k)\implies {\bf D}],
	\end{align*}
	where ${\bf D}=(\exists i\in[1,L])[(\forall s''\in L(\Scal))
		[(O_i(s'')=O_i(ss')) \implies (T_f\in s'')]]$.
\end{definition}

This notion means that whenever a faulty event occurs, in at least one location $i\in[1,L]$, the observer
$\Osf_i$ can make sure that after a common time delay (representing the number of
occurrences of events),
all generated event sequences with
the same observation
must contain a faulty event.

When $L=1$ and $O_1=\ell$, Definition \ref{def_codiag} reduces to the following notion 
of {\it diagnosability} \cite{Sampath1995DiagnosabilityDES,Jiang2001PolyAlgorithmDiagnosabilityDES,Yoo2002DiagnosabiliyDESPTime}:

\begin{definition}[Diag]\label{def_diag}
	An FSA $\Scal$ \eqref{DES_FSA} is called {\it $T_f$-diagnosable} if
	\begin{align*}
		&(\exists k\in\N)(\forall s\in L(\Scal)\cap T^*T_f)(\forall s':ss'\in L(\Scal))\\
		&[(|s'|\ge k)\implies {\bf D}],
	\end{align*}
	where ${\bf D}=(\forall s''\in L(\Scal))[(\ell(s'')=\ell(ss')) \implies (T_f\in s'')]$.
\end{definition}

We consider the following problem.

\begin{problem}[FSA CO-DIAGNOSABILITY]\label{prob_codiag}
	
	INSTANCE: An FSA $\Scal$ \eqref{DES_FSA}, a faulty event set $T_f\subset T$, and a set
	$\Osf=\{\Osf_i|i\in[1,L]\}$ of local observers.

	QUESTION: Is $\Scal$ $(\Osf,T_f)$-co-diagnosable?
\end{problem}

The size of the input of Problem \ref{prob_codiag} is $|\Scal|+|T_f|+\sum_{i=1}^{L}|T_i|\le
|\Scal|+(L+1)|T|$,
where $\Scal$ and $T_i$'s are the same as those in Problem \ref{prob_codet}.

\subsubsection{Equivalent condition}

Similarly to characterizing $\Osf$-co-detectability, we show an equivalent condition for
$(\Osf,T_f)$-co-diagnosability through characterizing its negation.

\begin{proposition}\label{prop1_codiag}
	An FSA $\Scal$ \eqref{DES_FSA} is not $(\Osf,T_f)$-co-diagnosable if and only if
	\begin{align*}
		&(\forall k\in\N)(\exists ss'\in L({\Scal}):s\in T^*T_f)\\
		&[(|s'|\ge k)\wedge\neg{\bf D}],
	\end{align*}
	where $\neg{\bf D}=(\forall i\in[1,L])[(\exists s''\in L(\Scal))
		[(O_i(s'')=O_i(ss')) \wedge (T_f\notin s'')]]$.
\end{proposition}

In order to verify $(\Osf,T_f)$-co-diagnosability of FSA $\Scal$ \eqref{DES_FSA}, we compute 
a simplified version 
\begin{equation}\label{ConCompDiagnosis_automata}
	\CCa(\Scal;\Scal_1^\nsf,\dots,\Scal_L^\nsf)
\end{equation}
of concurrent composition  \eqref{ConComp_automata}, where for each $i\in[1,L]$,
automaton $\Scal_i^\nsf$ is obtained from local automaton $\Scal_i$ \eqref{local_DES_FSA} by removing all
its faulty transitions, where we call $\Scal_i^\nsf$ a {\it normal sub-local automaton in location $i$}.

\begin{theorem}\label{thm1_codiag}
	An FSA $\Scal$ \eqref{DES_FSA} is not $(\Osf,T_f)$-co-diagnosable if and only if
	in 
	concurrent composition $\CCa(\Scal;\Scal_1^\nsf,\dots,\Scal_L^\nsf)=(X',T',X_0',\dt')$
	\eqref{ConCompDiagnosis_automata},
	where each $\Scal_i^\nsf$ is the normal sub-local automaton in location $i$, $i\in[1,L]$,
	\begin{subequations}\label{eqn1_codiag}
		\begin{align}
			&\text{there is a transition sequence}\\
			&x_0'\xrightarrow[]{s_1'}x_1'\xrightarrow[]{t'}x_2'\xrightarrow[]{s_{2}'}x_{3}'\xrightarrow[]{s_{3}'}x_{3}'\text{ such that}\label{eqn1b_codiag}\\
			&x_0'\in X_0', x_1',x_2',x_{3}'\in X', s_{1}',s_{2}',s_{3}'\in(T')^*,t'\in T';\\
			&t'(0)\in T_f;\quad |s_3'(0)|>0.
		\end{align}
	\end{subequations}
	
\end{theorem}

\begin{proof}
	``if'': Assume \eqref{eqn1_codiag} holds. For all $k\in\Z_{+}$,
	we choose $$s_1'(0)t'(0)s_2'(0)(s_3'(0))^k\in L(\Scal),$$ then we have a transition sequence
	\begin{equation}\label{eqn2_codiag}
		x_0'\xrightarrow[]{s_1'}x_1'\xrightarrow[]{t'}x_2'\xrightarrow[]{s_{2}'}x_{3}'\xrightarrow[]{s_{3}'}x_{3}'
		\xrightarrow[]{s_{3}'}\cdots\xrightarrow[]{s_{3}'}x_{3}',
	\end{equation}
	where $|s_3'(0)\dots s_3'(0)|\ge k$.
	Since $t'(0)$ is a faulty event, and there exists no faulty event in all components of 
	\eqref{eqn2_codiag} except for the $0$-th component, we have $\Scal$ is not $\Osf$-co-diagnosable 
	by Proposition \ref{prop1_codiag}.

	``only if'': This implication holds by Propositions \ref{prop1_co_DES}, \ref{prop1_codiag},
	the finiteness of states of $\Scal$, and the Pigeonhole Principle.
\end{proof}

\begin{remark}
	Dr. St\'{e}phane Lafortune from University of Michigan also found a so-called 
	verifier-based test for $(\Osf,T_f)$-co-diagnosability that is equivalent to
	the result shown in Theorem \ref{thm1_codiag} independently in unpublished course notes
	provided to the authors. In addition, Theorem \ref{thm1_codiag} is actually
	equivalent to the result in \cite{Moreira2011Codiagnosability}.
\end{remark}

\begin{example}\label{exam1_codiag}
	Consider the FSA $\Scal$ shown in Fig. \ref{fig2:codiag}, where the labeling function $\ell$ 
	is defined by $\ell(a)=a$, $\ell(b)=b$, $\ell(f)=\ell(u)=\ep$, only $f$ is faulty.
	Consider two local observers $\Osf_1$ and $\Osf_2$, where $a$
	can only be observed by $\Osf_1$, $b$ can only be observed by $\Osf_2$. 
	The local automaton corresponding to observer $\Osf_i$ is denoted by $\Scal_i$, $i=1,2$,
	where the corresponding $\Scal_1^{\nsf}$ and $\Scal_2^{\nsf}$ are shown in Fig. \ref{fig3:codiag}.
	Part of the concurrent composition $\CCa(\Scal;\Scal_1^{\nsf},\Scal_2^{\nsf})$ defined 
	by \eqref{ConCompDiagnosis_automata} is drawn in Fig. \ref{fig4:codiag}.

	\begin{figure}[htpb!]
			\tikzset{global scale/.style={
    scale=#1,
    every node/.append style={scale=#1}}}
		\begin{center}
			\begin{tikzpicture}[global scale = 1.0,
				>=stealth',shorten >=1pt,thick,auto,node distance=2.5 cm, scale = 1.0, transform shape,
	->,>=stealth,inner sep=2pt,
				every transition/.style={draw=red,fill=red,minimum width=1mm,minimum height=3.5mm},
				every place/.style={draw=blue,fill=blue!20,minimum size=7mm}]
				\tikzstyle{emptynode}=[inner sep=0,outer sep=0]
				\node[state, initial, initial where = left] (x0) {$x_0$};
				\node[state] (x1) [right of = x0] {$x_1$};
				\node[state] (x2) [below of = x1] {$x_2$};
				\node[state] (x3) [right of = x1] {$x_3$};
				\node[state] (x4) [below of = x3] {$x_4$};
				\node[state] (x5) [right of = x3] {$x_5$};

				\path[->]
				(x0) edge node [above, sloped] {$a,O_2(a)=\ep$} (x1)
				(x0) edge node [above, sloped] {$a,O_2(a)=\ep$} (x2)
				(x1) edge node [above, sloped] {$b,O_1(b)=\ep$} (x3)
				(x2) edge node [above, sloped] {$b,O_1(b)=\ep$} (x4)
				(x3) edge node [above, sloped] {$f$} (x5)
				(x5) edge [loop right] node {$u$} (x5)
				(x4) edge [loop right] node {$u$} (x4)
				;
			\end{tikzpicture}
	\end{center}
	\caption{FSA $\Scal$ with two local automata $\Osf_1$ and $\Osf_2$.}
	\label{fig2:codiag}
	\end{figure}

	\begin{figure}[htpb!]
			\tikzset{global scale/.style={
    scale=#1,
    every node/.append style={scale=#1}}}
		\begin{center}
			\begin{tikzpicture}[global scale = 1.0,
				>=stealth',shorten >=1pt,thick,auto,node distance=2.5 cm, scale = 1.0, transform shape,
	->,>=stealth,inner sep=2pt,
				every transition/.style={draw=red,fill=red,minimum width=1mm,minimum height=3.5mm},
				every place/.style={draw=blue,fill=blue!20,minimum size=7mm}]
				\tikzstyle{emptynode}=[inner sep=0,outer sep=0]
				\node[state, initial, initial where = left] (x0) {$x_0$};
				\node[state] (x1) [right of = x0] {$x_1$};
				\node[state] (x2) [below of = x1] {$x_2$};
				\node[state] (x3) [right of = x1] {$x_3$};
				\node[state] (x4) [below of = x3] {$x_4$};
				\node[state] (x5) [right of = x3] {$x_5$};

				\path[->]
				(x0) edge node [above, sloped] {$a$} (x1)
				(x0) edge node [above, sloped] {$a$} (x2)
				(x1) edge node [above, sloped] {$b(\ep)$} (x3)
				(x2) edge node [above, sloped] {$b(\ep)$} (x4)
				(x5) edge [loop right] node {$u$} (x5)
				(x4) edge [loop right] node {$u$} (x4)
				;

				\node[emptynode] (empnode1) [below of = x0] {} ;

				\node[state, initial, initial where = left] [below of = empnode1] (x0') {$x_0$};
				\node[state] (x1') [right of = x0'] {$x_1$};
				\node[state] (x2') [below of = x1'] {$x_2$};
				\node[state] (x3') [right of = x1'] {$x_3$};
				\node[state] (x4') [below of = x3'] {$x_4$};
				\node[state] (x5') [right of = x3'] {$x_5$};

				\path[->]
				(x0') edge node [above, sloped] {$a(\ep)$} (x1')
				(x0') edge node [above, sloped] {$a(\ep)$} (x2')
				(x1') edge node [above, sloped] {$b$} (x3')
				(x2') edge node [above, sloped] {$b$} (x4')
				(x5') edge [loop right] node {$u$} (x5')
				(x4') edge [loop right] node {$u$} (x4')
				;
			\end{tikzpicture}
	\end{center}
	\caption{Normal sub-local automata $\Scal_1^{\nsf}$ (up) and $\Scal_2^{\nsf}$ (below)
	corresponding to $\Scal$ in Fig.
	\ref{fig2:codiag}.}
	\label{fig3:codiag}
	\end{figure}

	\begin{figure}[htpb!]
			\tikzset{global scale/.style={
    scale=#1,
    every node/.append style={scale=#1}}}
		\begin{center}
			\begin{tikzpicture}[global scale = 1.0,
				>=stealth',shorten >=1pt,thick,auto,node distance=2.5 cm, scale = 1.0, transform shape,
	->,>=stealth,inner sep=2pt,
				every transition/.style={draw=red,fill=red,minimum width=1mm,minimum height=3.5mm},
				every place/.style={draw=blue,fill=blue!20,minimum size=7mm}]
				\tikzstyle{emptynode}=[inner sep=0,outer sep=0]
				\tikzstyle{stateCom}=[shape=rectangle, draw, thick, fill=gray!10]
				\node[stateCom, initial, initial where = left] (000) {$\begin{matrix}x_0\\ x_0\\ x_0\end{matrix}$};
				\node[stateCom] [right of = 000] (120) {$\begin{matrix}x_1\\ x_2\\ x_0\end{matrix}$};
				\node[stateCom] [right of = 120] (122) {$\begin{matrix}x_1\\ x_2\\ x_2\end{matrix}$};
				\node[stateCom] [right of = 122] (324) {$\begin{matrix}x_3\\ x_2\\ x_4\end{matrix}$};
				\node[stateCom] [right of = 324] (344) {$\begin{matrix}x_3\\ x_4\\ x_4\end{matrix}$};
				\node[stateCom] [above of = 344] (544) {$\begin{matrix}x_5\\ x_4\\ x_4\end{matrix}$};
				\node[stateCom] [above of = 120] (110) {$\begin{matrix}x_1\\ x_1\\ x_0\end{matrix}$};
				\node[stateCom] [right of = 110] (111) {$\begin{matrix}x_1\\ x_1\\ x_1\end{matrix}$};

				\path[->]
				(000) edge node {$(a,a,\ep)$} (120)
				(120) edge node {$(\ep,\ep,a)$} (122)
				(122) edge node {$(b,\ep,b)$} (324)
				(324) edge node {$(\ep,b,\ep)$} (344)
				(344) edge node [above, sloped] {$(f,\ep,\ep)$} (544)
				(544) edge [loop left] node [above, sloped] {$(u,\ep,\ep)$} (544)
				(000) edge node [above, sloped] {$(a,a,\ep)$} (110)
				(110) edge node [above, sloped] {$(\ep,\ep,a)$} (111)
				;

			\end{tikzpicture}
	\end{center}
	\caption{Part of concurrent composition $\CCa(\Scal;\Scal_1^{\nsf},\Scal_2^{\nsf})$, where $\Scal$
	is shown in Fig. \ref{fig2:codiag}, $\Scal_1^{\nsf},\Scal_2^{\nsf}$ are shown in Fig.
	\ref{fig3:codiag}.}
	\label{fig4:codiag}
	\end{figure}
	
	We verify its $(\{\Osf_1,\Osf_2\},\{f\})$-co-diagnosability by Theorem \ref{thm1_codiag}.
	In the concurrent composition $\CCa(\Scal;\Scal_1^{\nsf},\Scal_2^{\nsf})$,
	there is a transition sequence
	\begin{align*}
		&(x_0,x_0,x_0)\xrightarrow[]{(a,a,{\ep})}(x_1,x_2,{x_0})\xrightarrow[]{({\ep,\ep},a)}
		(x_1,x_2,x_2)\xrightarrow[]{(b,{\ep},b)}(x_3,{x_2},x_4)\xrightarrow[]{({\ep,b,\ep})}\\
		&(x_3,x_4,x_4)\xrightarrow[]{(f,{\ep,\ep})}
		(x_5,x_4,x_4)\xrightarrow[]{(u,{\ep,\ep})}(x_5,x_4,x_4)
	\end{align*} shown 
	in Fig. \ref{fig4:codiag} such that $(x_0,x_0,x_0)$ is initial, there is an event
	$(f,{\ep,\ep})$ whose $0$-th component is $f$, and after $(f,{\ep,\ep})$, there is a transition 
	cycle $(x_5,x_4,x_4)\xrightarrow[]{(u,{\ep,\ep})}(x_5,x_4,x_4)$ such that the $0$-th component
	of $(u,{\ep,\ep})$ is of positive length, that is, \eqref{eqn1_codiag} is satisfied. 
	Hence $\Scal$ is not $(\{\Osf_1,\Osf_2\},\{f\})$-co-diagnosable.
\end{example}

\begin{remark}\label{rem1:codiag}
	We want to point out that the co-diagnosability notion studied in 
	\cite{Qiu2006DecentralizedFD} is stronger than Definition \ref{def_codiag},
	here we call the former strong co-diagnosability. Strong co-diagnosability is
	defined by changing ``$(|s'|\ge k)$'' in Definition \ref{def_codiag} to
	``$((|s'|\ge k) \vee (ss'\text{ deadlocks}))$''.
	For an FSA \eqref{DES_FSA}, the technique of adding at each deadlock state a 
	self-loop containing an unobservable normal event, preserves the strong co-diagnosability,
	but does not always preserve Definition \ref{def_codiag}. And for an FSA,
	after modifying it in this way, the two definitions coincide. 
	Hence the method developed in 
	\cite{Qiu2006DecentralizedFD} works after doing this modification to FSAs,
	it does not apply to general FSAs with deadlock states.

	Consider the FSA $\Scal_1$ shown in Fig. \ref{fig1:codiag}.
	Choose $k=1$, then for the unique event sequence
	$f$ ended by a faulty event generated by $\Scal_1$, there is no continuation of $f$.
	Hence $\Scal_1$ satisfies Definition \ref{def_diag} vacuously. However, after we add a 
	self-loop on the unique deadlock state $x_2$ labeled by any non-faulty event of $\Scal_1$
	or a fresh (also non-faulty) event not in $\Scal_1$, then 
	it does not satisfy Definition \ref{def_diag} any more: For all $k\in\Z_{+}$, 
	choose event sequence $f\diamond^k$ generated by the modified FSA, where $\diamond$
	(non-faulty, could be equal to $u$)
	is the event of the newly added self-loop on $x_2$, $f$ is faulty, $|\diamond^k|\ge k$,
	event sequence $u\diamond^{k}$ generated by the modified FSA does not contain any faulty event
	but $u\diamond^{k}$ and $f\diamond^k$  produce the same output sequence $\ep$.
	\begin{figure}[htpb!]
				\tikzset{global scale/.style={
    scale=#1,
    every node/.append style={scale=#1}}}
		\begin{center}
			\begin{tikzpicture}[global scale = 1.0,
				>=stealth',shorten >=1pt,thick,auto,node distance=2.5 cm, scale = 1.0, transform shape,
	->,>=stealth,inner sep=2pt,
				every transition/.style={draw=red,fill=red,minimum width=1mm,minimum height=3.5mm},
				every place/.style={draw=blue,fill=blue!20,minimum size=7mm}]
				\tikzstyle{emptynode}=[inner sep=0,outer sep=0]
				\node[state] (x2) {$x_2$};
				\node[state, initial, initial where = above] (x1) [right of = x2] {$x_1$};
				\node[state, initial, initial where = above] (x0) [left of = x2] {$x_0$};

				\path[->]
				(x0) edge node [above, sloped] {$f$} (x2)
				(x1) edge node [above, sloped] {$u$} (x2)
				;

			\end{tikzpicture}
	\end{center}
	\caption{FSA $\Scal_1$, where all events are unobservable, only $f$ is faulty.}
	\label{fig1:codiag}
\end{figure}

	Similarly, one can see that there does exist a diagnosable FSA such that
	after adding at its
	deadlock state a self-loop labeled by a faulty unobservable event, the 
	modified FSA becomes no longer diagnosable, e.g., the FSA shown in Fig. \ref{fig5:codiag}.

		\begin{figure}[htpb!]
				\tikzset{global scale/.style={
    scale=#1,
    every node/.append style={scale=#1}}}
		\begin{center}
			\begin{tikzpicture}[global scale = 1.0,
				>=stealth',shorten >=1pt,thick,auto,node distance=2.5 cm, scale = 1.0, transform shape,
	->,>=stealth,inner sep=2pt,
				every transition/.style={draw=red,fill=red,minimum width=1mm,minimum height=3.5mm},
				every place/.style={draw=blue,fill=blue!20,minimum size=7mm}]
				\tikzstyle{emptynode}=[inner sep=0,outer sep=0]
				\node[state, initial, initial where = above] (x0) {$x_0$};
				\node[state] (x1) [right of = x0] {$x_1$};
				\node[state] (x2) [left of = x0] {$x_2$};

				\path[->]
				(x0) edge node [above, sloped] {$f$} (x1)
				(x0) edge node [above, sloped] {$u$} (x2)
				(x2) edge [loop left] node {$u$} (x2)
				;

			\end{tikzpicture}
	\end{center}
	\caption{An FSA, where all events are unobservable, only $f$ is faulty.}
	\label{fig5:codiag}
\end{figure}
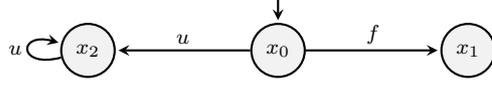
\end{remark}

\subsubsection{Complexity analysis}

\begin{theorem}\label{thm2_codiag}
	Problem \ref{prob_codiag} belongs to $\PSPACE$.
\end{theorem}

\begin{proof}
	We prove this conclusion by Theorem \ref{thm1_codiag}. In $\CCa(\Scal;\Scal_1^n,\dots,\Scal_L^n)$,
	guess states $x_1',x_2',x_3'$. Check (\romannumeral1) $x_1'$ is reachable, (\romannumeral2)
	there is a transition $x_1'\xrightarrow[]{t'}x_2'$ such that $t'(0)\in T_f$,
	(\romannumeral3) $x_3'$ equals $x_2'$ or $x_3'$ is reachable from $x_2'$, and (\romannumeral4)
	there is a transition cycle $x_3'\xrightarrow[]{s_3'}x_3'$ such that $|s_3'(0)|>0$, all by nondeterministic
	search. Then similarly to Theorem \ref{thm2_codet}, we also have $(\Osf,T_f)$-co-diagnosability 
	of $\Scal$ can be verified in $\PSPACE$.
\end{proof}

%

It was proved in \cite{Cassez2012ComplexityCodiagnosability} that Problem \ref{prob_codiag} is
$\PSPACE$-hard for deterministic FSAs. It was proved in \cite{Berard2018DiagnosabilityPetriNet}
that the problem of verifying $T_f$-diagnosability of FSA \eqref{DES_FSA}
is $\NL$-complete by using a variant of the twin-plant structure and linear temporal logic
and reducing the $\NL$-complete $\PATH$ problem to 
negation of $T_f$-diagnosability in logarithmic space,
where the variant is actually equivalent to $\CCa(\Scal;\Scal^{\nsf})$.

\subsection{Co-predictability}

\subsubsection{Formulation}

Given an FSA \eqref{DES_FSA}, a faulty event subset $T_f\subset T$, and 
a set $\Osf=\{\Osf_i|i\in[1,L]\}$ of local observers, the definition of $(\Osf,T_f)$-co-predictability
is formulated as follows.

\begin{definition}[CoPred]\label{def_copredic}
	An FSA $\Scal$ \eqref{DES_FSA} is called {\it $(\Osf,T_f)$-co-predictable} if
	\begin{align*}
		&(\exists k\in\N)(\forall s\in L(\Scal)\cap T^*T_f)(\exists s'\sqsubset s:T_f\notin s')\\
		&(\exists i\in[1,L])(\forall uv\in L(\Scal))\\
		&[((O_i(s')=O_i(u)) \wedge (T_f\notin u) \wedge (|v|\ge k)) \implies (T_f\in v)].
	\end{align*}
\end{definition}

This notion means that once a faulty event will definitely occur, then before any faulty event occurs,
at least one local observer 
can make sure that after a common time delay (representing the number of occurrences of events),
all generated event sequences with the same observation
without any faulty event must be continued by an event sequence containing a faulty event, so as to raise an alarm
to definite occurrence of some faulty event.

When $L=1$ and $O_1=\ell$, Definition \ref{def_copredic} reduces to the following notion 
of {\it predictability} \cite{Genc2009PredictabilityDES}:

\begin{definition}[Pred]\label{def_predic}
	An FSA $\Scal$ \eqref{DES_FSA} is called {\it predictable} if
	\begin{align*}
		&(\exists k\in\N)(\forall s\in L(\Scal)\cap T^*T_f)(\exists s'\sqsubset s:T_f\notin s')
		(\forall uv\in L(\Scal))\\
		&[((\ell(s')=\ell(u)) \wedge (T_f\notin u) \wedge (|v|\ge k)) \implies (T_f\in v)].
	\end{align*}
\end{definition}

We consider the following problem.

\begin{problem}[FSA CO-PREDICTABILITY]\label{prob_copred}
	
	INSTANCE: An FSA $\Scal$ \eqref{DES_FSA}, a faulty event set $T_f\subset T$, and a set
	$\Osf=\{\Osf_i|i\in[1,L]\}$ of local observers.

	QUESTION: Is $\Scal$ $(\Osf,T_f)$-co-predictable?
\end{problem}

The size of the input of Problem \ref{prob_copred} is the same as that of Problem 
\ref{prob_codiag}.

\subsubsection{Equivalent condition}

Similarly to characterizing the previous two properties, we still first characterize negation 
of $(\Osf,T_f)$-co-predictability in order to obtain its equivalent condition.

\begin{proposition}\label{prop1_copred}
	An FSA $\Scal$ \eqref{DES_FSA} is not $(\Osf,T_f)$-co-predictable if and only if
	\begin{align*}
		&(\forall k\in\N)(\exists s_k\in L({\Scal})\cap T^*T_f)(\forall s'\sqsubset s_k:T_f\notin s')\\
		&(\forall i\in[1,L])(\exists uv\in L(\Scal))\\
		&[(O_i(s')=O_i(u)) \wedge (T_f\notin uv) \wedge (|v|\ge k)].
	\end{align*}
\end{proposition}

In order to verify $(\Osf,T_f)$-co-predictability of FSA $\Scal$ \eqref{DES_FSA}, we need to compute
an even more simplified version 
\begin{equation}\label{ConCompPrediction_automata}
	\CCa(\Scal^{\nsf};\Scal_1^\nsf,\dots,\Scal_L^\nsf)
\end{equation} of concurrent composition \eqref{ConComp_automata},
where $\Scal^n$ is the
{\it normal sub-automaton} obtained from $\Scal$ by removing all its faulty transitions,
for each $i\in[1,L]$,
automaton $\Scal_i^\nsf$ is the previously defined  normal sub-local automaton in location $i$, $i\in[1,L]$.
Note that \eqref{ConCompPrediction_automata} is even simpler than \eqref{ConCompDiagnosis_automata}
that is used to characterize co-diagnosability.

\begin{theorem}\label{thm1_copred}
	An FSA $\Scal$ \eqref{DES_FSA} is not $(\Osf,T_f)$-co-predictable if and only if
	in 
	concurrent composition $\CCa(\Scal^{\nsf};\Scal_1^\nsf,\dots,\Scal_L^\nsf)=(X',T',X_0',\dt')$
	\eqref{ConCompPrediction_automata},
	where $\Scal^{\nsf}$ is the normal sub-automaton of $\Scal$, and  each $\Scal_i^\nsf$ is the 
	normal sub-local automaton in location $i$, $i\in[1,L]$,
	\begin{subequations}\label{eqn1_copred}
		\begin{align}
			&\text{there is a transition sequence}\\
			&x_0'\xrightarrow[]{s_1'}x_1'\text{ such that}\\
			&x_0'\in X_0', x_1'\in X', s_{1}'\in(T')^*;\\
			&(x_1'(0),t_f,x)\in\dt\text{ for some }t_f\in T_f\text{ and  }x\in X;\\
			&\text{for each }i\in[1,L],\text{ in }\Scal_i^\nsf,\text{ there is a transition cycle
			reachable from }x_1'(i).
		\end{align}
	\end{subequations}
\end{theorem}

\begin{proof}
	It directly follows from Propositions \ref{prop1_co_DES} and \ref{prop1_copred}.
\end{proof}

\begin{example}\label{exam1_copred}
	Reconsider the FSA $\Scal$ and the two local automata $\Scal_1,\Scal_2$
	shown in Fig. \ref{fig2:codiag}, and
	the corresponding $\Scal_1^{\nsf}$ and $\Scal_2^{\nsf}$ shown in Fig. \ref{fig3:codiag}.
	Observe that the concurrent composition $\CCa(\Scal^{\nsf};\Scal_1^{\nsf},\Scal_2^{\nsf})$
	defined by \eqref{ConCompPrediction_automata} could be obtained from 
	$\CCa(\Scal;\Scal_1^{\nsf},\Scal_2^{\nsf})$ by removing all transitions labeled by events
	$(f,*,*)$. Then from Fig. \ref{fig4:codiag}, one sees a reachable state $(x_3,x_4,x_4)$
	of $\CCa(\Scal^{\nsf};\Scal_1^{\nsf},\Scal_2^{\nsf})$ such that there is a faulty 
	transition $(x_3,f,x_5)$ in $\Scal$, and in $\Scal_i^{\nsf}$, $i=1,2$, there is a transition cycle
	$x_4\xrightarrow[]{u}x_4$ reachable from the $i$-th component $x_4$ of $(x_3,x_4,x_4)$.
	Then by Theorem \ref{thm1_copred}, $\Scal$ is not $(\{\Osf_1,\Osf_2\},\{f\})$-co-predictable.
\end{example}

\begin{remark}\label{rem1:copredic}
	We want to point out that technique used in \cite{Kumar2010CoPrognosisDES} of 
	adding at each deadlock state an unbound unobservable trace (e.g., an unobservable 
	self-loop) that is not observed by all local observers does not always preserve
	co-detectability.
	
	Consider the FSA $\Scal_2$ shown in Fig. \ref{fig1:copredic}.
	Choose $k=2$, then for the unique event sequence $f$ generated by $\Scal_2$
	ended by a faulty event, for $\ep\sqsubset
	f$, (1) choose $\ep u$, one has $|u|<2$, (2) choose $u\ep$, one also has $|\ep|<2$,
	hence $\Scal_2$ satisfies Definition \ref{def_predic} vacuously. However, if we add
	self-loops on $x_1$ and $x_2$ both labeled by $u$, then the modified $\Scal_2$ becomes 
	no longer predictable: For all $k\in\Z_{+}$, choose $f$, choose $\ep\sqsubset f$, the existence
	of event sequence $uu^k$ generated by the modified $\Scal_2$ violates Definition 
	\ref{def_predic}.

	\begin{figure}[htpb!]
				\tikzset{global scale/.style={
    scale=#1,
    every node/.append style={scale=#1}}}
		\begin{center}
			\begin{tikzpicture}[global scale = 1.0,
				>=stealth',shorten >=1pt,thick,auto,node distance=2.5 cm, scale = 1.0, transform shape,
	->,>=stealth,inner sep=2pt,
				every transition/.style={draw=red,fill=red,minimum width=1mm,minimum height=3.5mm},
				every place/.style={draw=blue,fill=blue!20,minimum size=7mm}]
				\tikzstyle{emptynode}=[inner sep=0,outer sep=0]
				\node[state, initial, initial where = above] (x0) {$x_0$};
				\node[state] (x1) [right of = x0] {$x_1$};
				\node[state] (x2) [left of = x0] {$x_2$};

				\path[->]
				(x0) edge node [above, sloped] {$f$} (x1)
				(x0) edge node [above, sloped] {$u$} (x2)
				;

			\end{tikzpicture}
	\end{center}
	\caption{FSA $\Scal_2$, where all events are unobservable, only $f$ is faulty.}
	\label{fig1:copredic}
\end{figure}

In the above modification, we add unobservable normal self-loops. It is easy to see if we
add a self-loop on $x_2$ labeled by $f$, then the predictability of $\Scal_2$ will be preserved.
Despite of this, we have the following example such that by adding at a 
deadlock state a self-loop labeled by an unobservable faulty event, predictability is not 
preserved. Consider $\Scal_3$ in Fig. \ref{fig2:copredic}. It is predictable vacuously since there
is no generated event sequence ended by a faulty event. However, if we add a self-loop
on the unique deadlock state $x_1$ labeled by a faulty event $f$, then the modified
$\Scal_3$ becomes no longer predictable.

	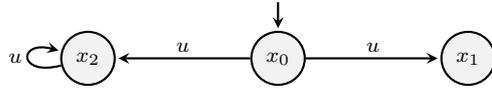
\begin{figure}[htpb!]
				\tikzset{global scale/.style={
    scale=#1,
    every node/.append style={scale=#1}}}
		\begin{center}
			\begin{tikzpicture}[global scale = 1.0,
				>=stealth',shorten >=1pt,thick,auto,node distance=2.5 cm, scale = 1.0, transform shape,
	->,>=stealth,inner sep=2pt,
				every transition/.style={draw=red,fill=red,minimum width=1mm,minimum height=3.5mm},
				every place/.style={draw=blue,fill=blue!20,minimum size=7mm}]
				\tikzstyle{emptynode}=[inner sep=0,outer sep=0]
				\node[state, initial, initial where = above] (x0) {$x_0$};
				\node[state] (x1) [right of = x0] {$x_1$};
				\node[state] (x2) [left of = x0] {$x_2$};

				\path[->]
				(x0) edge node [above, sloped] {$u$} (x1)
				(x0) edge node [above, sloped] {$u$} (x2)
				(x2) edge [loop left] node {$u$} (x2)
				;

			\end{tikzpicture}
	\end{center}
	\caption{FSA $\Scal_3$, where all events are normal and unobservable.}
	\label{fig2:copredic}
\end{figure}
\end{remark}

\subsubsection{Complexity analysis}

\begin{theorem}\label{thm2_copred}
	Problem \ref{prob_copred} belongs to $\PSPACE$.
\end{theorem}

\begin{proof}
	By Theorem \ref{thm1_copred}, in $\CCa(\Scal^{\nsf};\Scal_1^\nsf,\dots,\Scal_L^\nsf)$, we 
	guess a state $x_1'$. Check (\romannumeral1) $(x_1'(0),t_f,x)\in\dt$ for some $t_f\in T_f$ and $x\in X$,
	(\romannumeral2) for each $i\in[1,L]$, in $\Scal_i^n$, there is a transition cycle reachable from 
	$x_1'(i)$, all by nondeterministic search. Hence similarly to Theorem \ref{thm2_codiag}, 
	we also have $(\Osf,T_f)$-co-predictability of $\Scal$ can be verified in $\PSPACE$.
\end{proof}

\begin{corollary}
	$T_f$-predictability of FSA \eqref{DES_FSA} can be verified in $\PTIME$.
\end{corollary}

\begin{proof}
	The condition in Theorem \ref{thm1_copred} in case of $L=1$ and $O_1=\ell$ can be verified in 
	linear time in the size of $\CCa(\Scal^{\nsf};\Scal^\nsf)$,
	and $\CCa(\Scal^{\nsf};\Scal^\nsf)$ can be computed
	in quadratic polynomial time in the size of $\Scal$.
\end{proof}

\begin{theorem}\label{thm4_copred}
	The problem of verifying $T_f$-predictability of FSA \eqref{DES_FSA} is $\NL$-complete.
\end{theorem}

\begin{proof}
	We first show the $\NL$ membership. Guess states $x_1,\bar x_1,x_2,x_3$ of $X$.
	Check in $\CCa(\Scal^{\nsf};\Scal^{\nsf})$: (\romannumeral1) $(x_1,\bar x_1)$ is reachable;
	in $\Scal$:
	(\romannumeral2) $(x_1,f,x_2)\in\dt$ for some faulty event $t\in T_f$;
	and in $\Scal^{\nsf}$:
	(\romannumeral3) $x_3$ is equal to $\bar x_1$ or $x_3$ is reachable from $\bar x_1$,
	(\romannumeral4) $x_3$ belongs to a transition cycle, all by nondeterministic search.

	We can use the logspace reduction from the $\NL$-complete $\PATH$ problem
	to negation of diagnosability constructed in \cite{Berard2018DiagnosabilityPetriNet}
	to prove the $\NL$-hardness of verifying (negation of) predictability. 
	Given $G=(V,E)$ and $s,t\in V$, we define the FSA $\Scal_G=(V\cup\{v_f\},
	\{a,f,u\},\{s\},\dt_G,\{a\},\ell)$ as follows: (1) $v_f$ is a fresh state not in $V$,
	(2) $\dt_G=\{(v,a,v'),(v,a,v_f)|(v,v')\in E\}\cup\{(t,f,v_f),(t,u,v_f),(v_f,a,v_f)\}$, 
	(3) only $f$ is faulty, $\ell(f)=\ell(u)=\ep$, $\ell(a)=a$.
	We obtain an FSA $\Scal_G$ that satisfies Assumption \ref{assum1_Det_PN}.
	Then one sees that $\Scal_G$ is not predictable if and only if
	there is a directed path from $s$ to $t$ or $s=t$.
\end{proof}

In order to give a lower bound to co-predictability, we adopt the 
$\PSPACE$-complete DFA INTERSECTION problem
shown in \cite[Lemma 3.2.3]{Kozen1977FiniteAutomatonIntersection}.

\begin{proposition}[\cite{Kozen1977FiniteAutomatonIntersection}]\label{prop2_copred}
	Problem \ref{prob_DFAIntersec} is $\PSPACE$-complete if the input DFAs are all complete.
\end{proposition}

We need to change Proposition \ref{prop2_copred} slightly as follows.

\begin{proposition}\label{prop3_copred}
	Problem \ref{prob_DFAIntersec} is $\PSPACE$-hard if the input 
	DFAs $\A_1,\dots,\A_n$ satisfy that $\A_1$ has exactly one accepting state
	and the accepting state is deadlock, and all other $\A_i$'s are deadlock-free
	and have all states accepting.
\end{proposition}

\begin{proof}
	We are given complete DFAs $\A_1,\dots,A_n$ over the same alphabet $\Sig$.
	Construct DFA $\A'_1$ from $\A_1$ by adding
	transitions $q\xrightarrow[]{\lambda}\diamond_1$
	at each accepting state $q$, 
	changing all accepting states to be
	non-accepting, and changing $\diamond_1$ to be accepting.
	For each $2\le i\le n$, construct deadlock-free DFA $\A'_i$ from $\A_i$ by adding
	transitions $q\xrightarrow[]{\lambda}\diamond_i$
	at each accepting state $q$, also adding self-loop on $\diamond_i$ labeled by $\lambda$,
	changing all non-accepting states, including $\diamond_i$,
	to be accepting. Then $\A_1'$ has exactly one accepting state $\diamond_1$,
	$\diamond_1$ is the unique deadlock state of $\A_1'$, and $\A_2',\dots,\A_n'$
	have all states accepting and deadlock.
	One sees that for each word $w\in\Sig^*$, $w$ is accepted 
	by all $\A_i$'s if and only if $w\lambda$ is accepted by all $\A_1',\dots,\A_n'$.
	By Proposition \ref{prop2_copred}, this proposition holds.
\end{proof}

Next we give a lower bound for co-predictability, where the reduction is inspired from 
the $\PSPACE$-hardness proof of verifying co-diagnosability of deterministic FSAs
\cite{Cassez2012ComplexityCodiagnosability}.

\begin{theorem}\label{thm3_copred}
	Problem \ref{prob_copred} is $\PSPACE$-hard for deterministic deadlock-free FSAs.
\end{theorem}

\begin{proof}
	We prove this result by Proposition \ref{prop3_copred}.

	We are given DFAs $\A_0,\dots,\A_n$ over the same alphabet $\Sig$
	such that $\A_0$ has exactly one accepting state and the accepting state is deadlock
	and all other $\A_i$'s have 
	all states accepting and deadlock-free.
	Next we construct an FSA $\Scal$ from $\A_0,\dots,\A_n$ in polynomial time
	as shown in Fig. \ref{fig3:copred}.
	In each $\A_i$, $i\in[0,L]$, change each letter (i.e., event) $\s\in\Sig$ to $\s_i$,
	and set $\Sig_i=\{\s_i|\s\in\Sig\}$ and $\ell(\s_i)=\s$, where $\ell$ is the labeling function. 
	Add initial state $\diamond_0$ and transition $\diamond_0\xrightarrow[]{a_i}q$ to the initial state $q$
	of each $\A_i$, $i\in[0,L]$, where $\diamond_0$ differs from any state of any $\A_i$,
	$a_i\notin\bigcup_{i\in[0,L]} \Sig_i$, and set $\ell(a_i)=a$. 
	Change each initial state of each $\A_i$ to be non-initial,
	$i\in[0,L]$. Add state $\diamond_1$ that is not any state of any $\A_i$,
	$i\in[0,L]$, and self-loop on it labeled by a fresh event $a$, and set $\ell(a)=a$.
	At the accepting state of $\A_0$, add transition to $\diamond_1$ labeled by event $F$,
	where $F$ is a new event and set to be faulty. We also set $\ell(F)=\ep$, i.e., $F$ is unobservable.
	We have obtained the deterministic deadlock-free FSA $\Scal$ with a unique initial state.
	In $\Scal$, all states except for the unique accepting state of $\A_0$ can be reachable 
	from some state $q$ through only non-faulty transitions such that $s$ belongs to a non-faulty
	transition cycle.
	The number of states of $\Scal$ equals the sum of numbers of states of all $\A_i$, $i\in[0,L]$,
	plus $2$ (corresponding to newly added states $\diamond_1,\diamond_2$).
	The event set of $\Scal$ is $\bigcup_{i\in[0,L]}\Sig_i\cup\{a,F\}\cup\{a_i|i\in[0,L]\}$.

	\begin{figure}[htbp]
		\tikzset{global scale/.style={
    scale=#1,
    every node/.append style={scale=#1}}}
		\begin{center}
			\begin{tikzpicture}[global scale = 1.0,
				>=stealth',shorten >=1pt,thick,auto,node distance=2.5 cm, scale = 1.0, transform shape,
	->,>=stealth,inner sep=2pt,
				every transition/.style={draw=red,fill=red,minimum width=1mm,minimum height=3.5mm},
				every place/.style={draw=blue,fill=blue!20,minimum size=7mm}]
				\tikzstyle{emptynode}=[inner sep=0,outer sep=0]
				\node[state, initial] (dia0) {$\diamond_0$};
				\node[state] (s00) [above right of = dia0] {};
				\node[state, accepting] (s0f) [right of = s00] {};
				\node[state, accepting] (sif) [below right of = dia0] {};
				\node[state, accepting] (sif') [right of = sif] {};
				\node[state] (dia1) [right of = s0f] {$\diamond_1$};

				\path[->]
				(dia0) edge node [above, sloped] {$a_0(a)$} (s00)
				(dia0) edge node [above, sloped] {$a_i(a)$} (sif)
				(s0f) edge node [above, sloped] {$F$} (dia1)
				(dia1) edge [loop right] node {$a$} (dia1)
				;

				\draw[dashed] (3.0,1.8) ellipse (2cm and 1cm);
				\node at (3.0,1.8) {$\A_0$};

				\draw[dashed] (3.0,-1.8) ellipse (2cm and 1cm);
				\node at (3.0,-1.2) {$\A_i,1\le i\le n$};

			\end{tikzpicture}
			\caption{Sketch of the reduction in the proof of Theorem \ref{thm3_copred}.}
			\label{fig3:copred}
		\end{center}
	\end{figure}
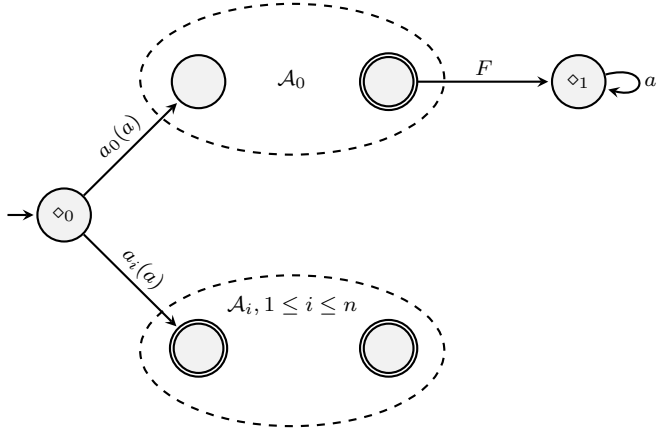

	Now we specify local automaton $\Scal_i$, $i\in[1,L]$, we only need to specify the labeling function $O_i$
	of each $\Scal_i$. For each $i\in[1,L]$, the observable event set is
	$\Sig_0\cup\Sig_i\cup\{a\}\cup\{a_i|i\in[0,L]\}$.
	That is, at each location $i\in[1,L]$, observer $\Osf_i$ can observe $\A_0$, $\A_i$, and all
	transitions outside $\bigcup_{j\in[0,L]}\A_j$ except for the faulty transitions. 

	In order to prove this theorem, we only need to prove that
	there is a word $w\in\Sig^*$ that is accepted by all $\A_i$, $i\in[0,L]$, if and only if,
	$\Scal$ is not $(\Osf,T_f)$-co-predictable.

	$\Rightarrow$: Assume $w=w^1\dots w^n\in\Sig^*$ that is accepted by all $\A_i$, $i\in[0,L]$,
	where $w^1,\dots,w^n\in\Sig$, $n\in\N$.
	Then $a_0w^1_0\dots w^n_0F\in L(\Scal)$ and $a_iw^1_i\dots w^n_i\in L(\Scal)$ for all $i\in[1,L]$.
	For each $i\in[1,L]$, $O_i(a_iw^1_i\dots w^n_i)=O_i(a_0w^1_0\dots w^n_0F)=aw$.
	Since $\Scal$ is deadlock-free and only $F$ is 
	faulty, we have $\Scal$ is not $(\Osf,T_f)$-co-predictable by Proposition \ref{prop1_copred}.

	$\Leftarrow$:
	Assume that $\Scal$ is not $(\Osf,T_f)$-co-predictable. Then by Theorem \ref{thm1_copred}
	and the structure of $\Scal$ (there is a unique faulty transition and the transition
	starts at the unique accepting state of $\A_0$),
	there exist $a_{\iota}w_{\iota}F\in L(\Scal)$ , $a_{j_i}w_{j_i}\in L(\Scal)$ for each
	$i\in[1,L]$ such that $O_i(a_{\iota}w_{\iota}F)=O_i(a_{j_i}w_{j_i})$, and $a_{j_i}w_{j_i}$
	does not lead $\Scal$ to the unique accepting state of $\A_0$. 
	One must have $\iota=0$ since $F$ can only follow words in $\Sig_0^*$, then
	$w_\iota\in\Sig_0^*$. Assume $w_{\iota}=\ep$. Because $\ell(w_{\iota})=\ep$ leads $\A_0$ to its
	accepting
	state (otherwise $w_{\iota}$ cannot be continued by $F$), $\ep$ is accepted by $\A_0$. $\ep$ is 
	always accepted by all other $\A_i$, $i\in[1,L]$. Next assume $w_{\iota}\ne\ep$. Then one has
	$\ell(w_{\iota})$ is accepted by $\A_0$, and then $j_i=i$ for all $i\in[1,L]$, otherwise
	there exists $k\in[1,L]$ such that $a_{j_k}w_{j_k}$ leads $\Scal$ to the unique accepting state of $\A_0$,
	reaching a contradiction. Note that
	$\ell(w_{j_i})$ is accepted by $\A_i$ for each $i\in[1,L]$ since all such $\A_i$'s have all states
	accepting, completing the proof.
\end{proof}

\section{Conclusion}\label{sec4:conc}

In this paper, we proposed a decentralized version of strong detectability of 
FSAs, and gave a $\PSPACE$ upper bound and a $\coNP$ lower bound for the notion.
In addition, we gave a unified concurrent-composition method to verify decentralized versions
of strong detectability,
diagnosability, and predictability of FSAs without any assumption or changing the FSAs under
consideration, which actually reveals essential relationships between these notions.

Moreover, the unified method could provide synthesis algorithms for
enforcing the three properties by choosing to disable several controllable transitions.
By Theorems \ref{thm1_codet}, \ref{thm1_codiag}, and \ref{thm1_copred}, lacks of the three
properties are due to existence of special transition sequences in the corresponding
concurrent compositions, the synthesis algorithms are to choose to disable several
controllable events of FSAs in order to remove all such special transition sequences.
These processes can be done in $\PTIME$ in 
the size of the corresponding concurrent compositions. 

This unified method could also be applied to characterize other variants of co-diagnosability
or co-predictability in the literature, or applied to characterize more general distributed 
versions of these fundamental properties under weak assumptions on the underlying networks.

\end{document}